\documentclass[12pt,reqno]{amsart}

\usepackage{verbatim}
\usepackage{amsmath,amsfonts,amscd,amssymb,epsfig,euscript,bm}
\usepackage{amsxtra,mathrsfs}
\usepackage{pdfsync}
\newtheorem{theorem}{Theorem}[section]

\newtheorem{lemma}{Lemma}[section]
\newtheorem{corollary}[theorem]{Corollary}

\theoremstyle{remark}
\newtheorem{remark}[theorem]{Remark}
\newtheorem{definition}[theorem]{Definition}

\newcommand{\pbp}{\mathcal{\partial \overline{\partial}}}
\newcommand{\spbp}{\sqrt{-1}\partial \overline{\partial}}
\newcommand{\ddc}{\frac{\sqrt{-1}}{2\pi}\mathcal{\partial \overline{\partial}}}

\numberwithin{equation}{section}
\usepackage{hyperref}
\hypersetup{colorlinks,citecolor=blue,plainpages=false,hypertexnames=false}
\begin{document}

\title[Metric properties of parabolic ample bundles]{Metric properties of 
parabolic ample bundles}

\author[I. Biswas]{Indranil Biswas}
\address{School of Mathematics,
Tata Institute of fundamental research, Homi Bhabha road, Mumbai 400005, India,
and Mathematics Department, EISTI-University Paris-Seine, Avenue du parc, 95000,
Cergy-Pontoise, France}
\email{indranil@math.tifr.res.in}
\author{Vamsi Pritham Pingali}
\address{Department of Mathematics, Indian Institute of Science, Bangalore 560012, India}
\email{vamsipingali@iisc.ac.in}



\begin{abstract}
We introduce a notion of admissible Hermitian metrics on parabolic bundles and define positivity properties for
the same. We develop Chern--Weil theory for parabolic bundles and prove that our metric notions coincide with the 
already existing algebro--geometric versions of parabolic
Chern classes. We also formulate a Griffiths conjecture in the parabolic setting 
and prove some results that provide evidence in its favour for certain kinds of parabolic bundles. For these kinds of parabolic structures, we prove that the conjecture holds on Riemann 
surfaces. We also prove that a Berndtsson--type result holds, and that there are metrics on
stable bundles over surfaces whose Schur forms are positive.
\end{abstract}

\maketitle

\section{Introduction}\label{Introsec}

Given a Hermitian holomorphic vector bundle $(E,H)$ on a complex manifold $X$, it is said to 
be Griffiths (respectively, Nakano) positive if the curvature $\Theta_H$ is a positive 
bilinear form when tested against $v\otimes s$ where $v$ is a tangent vector and $s$ a 
vector from $E$ (respectively, when tested against all vectors in $TX \otimes E$). Another 
notion of positivity is Hartshorne ampleness --- a holomorphic vector bundle $E$ is 
Hartshorne ample if the tautological line bundle $\mathcal{O}_{\mathbb{P}(E)}(1)$ over 
$\mathbb{P}(E)$ is ample in the usual sense. It is clear that a Griffiths positive bundle is 
Hartshorne ample. The converse is a well known conjecture of Griffiths.\\

The evidence available in favour of Griffiths' conjecture is as follows :
\begin{enumerate}
\item Mori, \cite{Mori}, proved Hartshorne's conjecture \cite{Hartshorne}. This means that a
compact complex manifold $M$ whose tangent bundle $TM$ is Hartshorne ample is biholomorphic
to $\mathbb{CP}^n$. Since the Fubini--study metric on $\mathbb{CP}^n$ has positive bisectional
curvature, the vector bundle $TM$ is Griffiths positive.

\item Umemura, \cite{Umemura}, and later, Campana and Flenner, \cite{Camp}, proved the Griffiths 
conjecture for Riemann surfaces.

\item Bloch and Gieseker, \cite{BG}, proved that the Schur polynomials of Hartshorne ample 
bundles are numerically positive. Griffiths himself proved that $c_1$ and $c_2$ of a 
Griffiths positive metric are positive as forms. Guler, \cite{Guler}, and Diverio, 
\cite{Diverio}, proved, using a complicated calculation based on an elegant idea of Guler,
that the signed Segre forms of Griffiths positive bundles are positive (in particular, on 
surfaces the Schur polynomials of a Griffiths positive metric are positive pointwise). In 
\cite{Pinchern} it was proven that a Hartshorne ample semistable bundle over a surface admits 
a metric whose Schur polynomials are positive pointwise. It is still unknown as to whether 
Schur polynomials of Griffiths positive metrics are pointwise positive, however if so, this would 
be further indirect evidence.

\item Demailly, \cite{Dem}, proved that if $E$ is Griffiths positive then $E\otimes \det(E)$ is 
Nakano positive. Berndtsson \cite{Bo} proved that if $E$ is Hartshorne ample, then $E\otimes 
\det(E)$ is Nakano positive. Mourougane--Takayama \cite{MT} independently proved that 
$E\otimes \det(E)$ is Griffiths positive if $E$ is Hartshorne ample.

\item Typically, ``good'' metrics are produced using flows. Naumann, \cite{Naumann}, outlined a 
promising approach to the Griffiths conjecture using the relative K\"ahler--Ricci flow. If it 
works, it ought to work just as well in the equivariant context (which is roughly what this 
paper deals with).
\end{enumerate}

It is but natural to wonder if the same kind of a conjecture can be made for 
singular Hermitian metrics. Unfortunately, the notion of a singular Hermitian 
metric on general vector bundles (as opposed to line bundles where a lot of work 
has been done) is quite subtle and only recently has there been progress on it 
\cite{Decataldo, BernPaun, Pauntakayama, Raufi, Raufinakano, Bernd, RaufiChern, 
Hosono}. A compromise can be made by choosing to work with parabolic bundles, 
which are essentially vector bundles equipped with flags (and weights) over 
divisors. Any reasonable notion of a ``metric" on a parabolic bundle should 
degenerate on the divisor, i.e., it should be a singular Hermitian metric. The 
differential geometry of parabolic bundles has been studied reasonably well \cite{Si},
\cite{Biq}, \cite{Po}, \cite{JayLu}. The notion of parabolic Hartshorne ampleness has also been studied \cite{Bi, BL, BN}. However, to our knowledge, the metric aspects of parabolic ampleness have not received any attention so far. This paper attempts to remedy that situation.\\

In this paper we prove the following results.
\begin{enumerate}
\item In Section \ref{Admissible} we introduce a notion of admissible Hermitian metrics on parabolic bundles with rational weights over projective manifolds. It is interesting to compare our definition of admissibility with existing ones. We plan on addressing this in future work.

\item We define a metric notion of Griffiths (and Nakano) positivity for parabolic bundles in 
Section \ref{positivity-and-ampleness} and formulate a Griffiths conjecture in this context. 
We prove it for Riemann surfaces (for certain kinds of parabolic structures induced from 
``good Kawamata covers''). Moreover we prove that our notion of positivity agrees with the 
algebro--geometric notion in \cite{Bi} for line bundles.

\item In Section \ref{ChernWeil} we develop Chern--Weil theory for admissible metrics on parabolic bundles. We verify 
that the Chern classes coincide with the ones defined algebraically in \cite{Bi3, IS, BD}. We prove that the push--forward 
of $c_1^k(\mathcal{O}_{\mathbb{P}(E_*)}(1))$ gives (signed) Segre forms of $E$. This is a parabolic version of some 
results in \cite{Guler, Diverio}. Our proof has a small technical innovation in terms of generating functions and we 
hope it generalizes to computing push--forwards for flag bundles. Lastly, we prove a parabolic version of a result (for parabolic structures arising from good Kawamata covers) in 
\cite{Pinchern} concerning the existence of metrics whose Schur forms are positive on stable bundles over surfaces.

\item In Section \ref{Directimages} we prove a parabolic analogue of Berndtsson's theorem (as above, for parabolic structures arising from good Kawamata covers), i.e., if $E$ is Hartshorne ample, $E \otimes \det(E)$ is Nakano positive.
\end{enumerate}

\section{Preliminaries}\label{Prelim}

\subsection{Definition of parabolic vector bundles}

Let $X$ be an irreducible smooth complex projective variety and $D\, \subset\, X$ a reduced effective simple 
normal crossing divisor; this means that for the decomposition
$$D\,=\, \displaystyle \sum_{i=1}^{\mu} D_i$$ into 
irreducible components, each component $D_i$ is smooth and they intersect transversally. In this paper we will 
state and prove results only for the case of $\mu\,=\,1$, i.e., for smooth divisors. The general case of simple 
normal crossings is not such a big leap from our current study.

\begin{definition}
Let $E$ be a holomorphic vector bundle on $X$ of rank $r$. A quasi-parabolic structure on $E$ over $D$
is a filtration
\begin{equation}\label{filt}
E\vert_{D_i} \,=\, F_1^i \,\supsetneq\, F_2 ^i\,\supsetneq\, \cdots \,\supsetneq\,
F^i_{m_i} \,\supsetneq\, F^i_{m_{i+1}}\, =\, 0\, ,
\end{equation}
where each $ F_j^i$ is a subbundle of $E\vert_{D_i}$ such that they are locally
abelian, which means that for every $x\, \in\, D$ there is a decomposition of
$E_x$ into a direct sum of lines with the property that for any $i$ with $x\, \in\, D_i$, the
filtration of $E\vert_{D_i}$ when restricted to $x$, is given by combinations of these lines. Note that
when $\mu\,=\, 1$, this condition of being locally abelian is automatically satisfied.

A \emph{parabolic structure} is a quasi-parabolic structure as in
\eqref{filt} endowed with \emph{parabolic weights} which are 
collections of rational numbers
$$0\,<\, \alpha_1^i \,\leq\,\alpha_2^i\,\leq\, \,\cdots\, \leq\,\alpha_{r}^i\,<\,1\, ,$$
(where $\alpha_j^i$ can be repeated) associated to the subbundles $F^i_k$, i.e., $\alpha^i_1
\,=\,\ldots\,=\,\alpha^i_{r_{i,1}}$ correspond to $F^i_{1}/F^i_2$, etc where $r_{i,j}$ is the
rank of $F^i_j/F^i_{j+1}$; more precisely, $$\alpha^i_{1+\sum_{j=1}^a r_{i,j}}\,=\,
\alpha^i_{2+\sum_{j=1}^a r_{i,j}}\, =\, \cdots \,=\,
\alpha^i_{\sum_{j=1}^{a+1} r_{i,j}}$$
correspond to $F^i_{a+1}/F^i_{a+2}$ for all $1\, \leq\, a\, \leq\, m_i-1$,
and this common number is called the weight of $F^i_{a+1}/F^i_{a+2}$. A parabolic vector bundle
is one that is equipped with a parabolic
structure. For notational convenience, a parabolic vector bundle $(E,\, \{F^i_j\},\, {\alpha}^i_j)$
will also be denoted as $E_{*}$. The divisor $D$ is called the parabolic divisor for $E_{*}$. 
\label{parabolicdef}
\end{definition}

\begin{remark}
Note that if all the parabolic weights are zero (the ``trivial parabolic structure''), we do 
not call it a parabolic bundle in this paper.
\end{remark}

Take a parabolic vector bundle $E_*$. Maruyama and Yokogawa associate to $E_*$ a filtration of 
coherent sheaves $\{E_t\}_{t\in \mathbb{R}}$ parametrized by $\mathbb{R}$ \cite{MY}. This 
filtration encodes the entire parabolic data. We recall from \cite{MY} some properties of this 
filtration:
\begin{enumerate}
\item the filtration $\{E_t\}_{t\in \mathbb{R}}$ is decreasing as $t$ increases,
meaning $E_{t+t'}\, \subset\, E_t$ for all $t'\, >\, 0$ and $t$,

\item it is left--continuous, meaning for all $t\,\in\, \mathbb{R}$, there is
$\epsilon_t\, >\, 0$ such that the above inclusion of $E_t$ in $E_{t-\epsilon_t}$
is an isomorphism,

\item $E_{t+1}\,=\, E_t\otimes {\mathcal O}_X(-D)$ for all $t$,

\item{} the vector bundle $E$ is $E_0$,

\item for an finite interval $[a,\, a']$, the set
$$
\{t\,\in\, [a,\, a']\,\mid\, E_{t+\epsilon}\, \subsetneq\, E_t~\ \forall \, ~
\epsilon\, >\, 0\}
$$
is finite, and

\item the filtration $\{E_t\}_{t\in \mathbb{R}}$ has a right jump at $t$ if and
only if $t-[t]$ is a parabolic weight for $E_*$.
\end{enumerate}

Fix a very ample line bundle on $X$ to define \textit{degree} of coherent sheaves on $X$.
The parabolic degree of a parabolic bundle $E_*$ as above is defined to be
$$
\text{par-deg}(E_*)\,:=\, \text{degree}(E)+\sum_{i=1}^\mu \sum_{j=1}^{m_i}
\text{degree}(F^i_j/F^i_{j+1})\cdot \text{weight}(F^i_j/F^i_{j+1})\, .
$$
In terms of the filtration $\{E_t\}_{t\in \mathbb{R}}$, we have
$$
\text{par-deg}(E_*)\,=\,r\cdot \text{degree}({\mathcal O}_X(D))+
\int_0^1 \text{degree}(E_t)dt\, .
$$

Now we will recall the definitions of direct sum, tensor product and dual of
parabolic vector bundles.

Let $E_*$ and $V_*$ be parabolic vector bundles with a common parabolic divisor $D$.
The underlying vector bundles for $E_*$ and $V_*$ will be denoted by $E$ and $V$
respectively. Let
$$
\iota\, :\, X\setminus D\, \hookrightarrow\, X
$$
be the inclusion map. Consider the quasi--coherent sheaf $\iota_*\iota^* (E\oplus V)$
on $X$. The
parabolic direct sum $E_*\oplus F_*$ is defined to be the parabolic vector bundle that
corresponds to the filtration $\{E_t\oplus F_t\}_{t\in \mathbb{R}}$ of subsheaves of it.

Next consider the quasi--coherent sheaf $\iota_*\iota^* (E\otimes V)$ on $X$. For any $t\, \in\,
\mathbb{R}$, let $U_t$ be the coherent subsheaf of it generated by
all $E_{s}\otimes V_{t-s}$, $s\, \in\, \mathbb{R}$. The conditions on
$\{E_{b}\}_{b\in \mathbb{R}}$ and $\{V_{b}\}_{b\in \mathbb{R}}$ ensure that this
$U_t$ is indeed a coherent sheaf. The parabolic 
tensor product $E_*\otimes F_*$ is defined to be the parabolic vector bundle that 
corresponds to this filtration $\{U_t\}_{t\in \mathbb{R}}$.

For any $t\, \in\, \mathbb{R}$, define $E_{t+}$ to be $E_{t+\epsilon}$, where $\epsilon
\,>\, 0$ is sufficiently small so that $E_{t+\epsilon}$ is independent of
$\epsilon$ (recall that the filtration parametrized by $\mathbb R$ has finitely
many jumps in each bounded interval so it is constant except for those finitely
many jumps). Therefore, $(E_{-t-1+\epsilon})^*$ is a subsheaf of $\iota_*\iota^* E^*$. The
parabolic dual $E^*_*$ of $E_*$ is defined by the filtration
$\{(E_{-t-1+\epsilon})^*\}_{t\in \mathbb{R}}$. So the underlying vector bundle for the
parabolic dual $E^*_*$ is $(E_{\epsilon-1})^*$.

\subsection{Parabolic bundles and equivariant bundles}

Let $Y$ be a connected smooth complex projective variety and
$$
\Gamma\, \subset\, \text{Aut}(Y)
$$
a finite subgroup of the group of automorphisms of the variety $Y$. A
$\Gamma$--linearized vector bundle over $Y$ of rank $r$ is a holomorphic vector bundle $V$ of
rank $r$ over $Y$ equipped with a holomorphic action of $\Gamma$ such that
\begin{itemize}
\item the projection $V\, \longrightarrow\, Y$ is $\Gamma$--equivariant, and

\item the action of $\Gamma$ on $V$ is fiberwise linear.
\end{itemize}
In other words, $V$ is an orbifold vector bundle; it is also called an
equivariant bundle.

For any point $y\, \in\, Y$, let $\Gamma_y\, \subset\, \Gamma$ be the
isotropy subgroup of $y$ for the action of $\Gamma$ on $Y$.

Assume that that quotient variety $Y/\Gamma$ is smooth. 
Let $$q\, :\, Y\, \longrightarrow\, 
Y/\Gamma$$ be the quotient map. Consider the ramification divisor for $q$; let $$D_q\,\subset\,
Y$$ be the reduced ramification divisor for $q$. We assume that $D_q$ 
is a normal crossing divisor of $Y$.

Take a $\Gamma$--linearized vector bundle $V$ on $Y$. Let 
$\widetilde{D}\, \subset\, D_q$ be the union of all the irreducible components $D'$ of $D_q$ 
with the property that the isotropy subgroup $\Gamma_z$ of every point $z$ of $D'$ acts 
nontrivially on the fiber $V_z$ of $V$ over $z$. By means of the invariant direct image 
construction, $V$ produces a parabolic vector bundle $E_{*}$ on $Y/\Gamma$ with parabolic 
structure over the divisor $q(\widetilde{D})$ \cite{Bi1}, \cite{Bo1}, \cite{Bo2}.

Conversely, given a parabolic vector bundle $E_*$ on $X$ with parabolic structure over
a simple normal crossing divisor $D$, there is a triple $(Y,\, \Gamma, \, V)$ as above
such that
\begin{itemize}
\item $X\,=\, Y/\Gamma$, and

\item $E_*$ coincides with the parabolic vector bundle over $X$ associated to $V$ \cite{Bi1},
\cite{Bo1}, \cite{Bo2}.
\end{itemize}
This covering $Y$ is an example of
``Kawamata covering'' introduced by Kawamata \cite[Theorem 17]{Ka}, \cite[Theorem 1.1.1]{KMM}
in order to prove what is known as Kawamata--Viehweg vanishing theorem.
It should be clarified that the ramification divisor of the above quotient map
$$q \, :\, Y\, \longrightarrow\, X\, =\, Y/\Gamma$$ is in general bigger than $D$.
However on the inverse image $q^{-1}(X\setminus D) \, \subset\, Y$, the vector bundle $V$ is
canonically $\Gamma$--equivariantly
identified with the pullback $q^*(E\vert_{X\setminus D})$ (note that the pulled back bundle
$q^*(E\vert_{X\setminus D})$ has a tautological action of $\Gamma$); in other words, $E\vert_{X\setminus D}$ is the
descent of $V\vert_{q^{-1}(X\setminus D)}$. In particular, for any point $y\, \in\,
q^{-1}(X\setminus D)$, the action on the fiber $V_y$, of the equivariant vector
bundle $V$, of the isotropy subgroup $\Gamma_y$ is trivial.

The above correspondence between the parabolic vector bundles and the orbifold 
vector bundles is compatible with the operations of direct sum, tensor product, 
dualization etc. 

In view of the above, we make the following definitions.

\begin{definition}
Suppose $X$ is a complex manifold and $D\subset X$ is a divisor whose 
components are smooth and intersect transversally. Assume that $(E_{*},\,D)$ is a parabolic 
vector bundle on $X$. A triple $(Y,\,q,\,V)$ is called a \emph{Kawamata cover} of $(X,E_*,\,D)$ if
the following two conditions are satisfied:
\begin{itemize}
\item $q\,:\,Y\,\longrightarrow \,X$ is a finite branched cover of $X$ whose ramification divisor
$$D''\,=\,D \cup D'$$ has smooth transversally intersecting irreducible components, and

\item $E_*$
coincides with the bundle obtained by the invariant direct image construction
of the equivariant vector bundle $V$ over $Y$ equipped with an action of the covering group
$\text{Gal}(q)$ for $q$.
\end{itemize}
Such a Kawamata cover $(Y,\,q,\,V)$ is called \emph{good} if $D''\,=\,D$. Also, A Kawamata cover
$(Y,\,q,\,V)$ is called \emph{locally good around
a point} $p\,\in\, X$ if there is a Zariski open
neighbourhood $p\,\in\, U_p \,\subset\, X$ such that $D'' \cap U_p \,=\, D\cap U_p$. A
Kawamata cover of $(X,E_*,\,D)$ is called \emph{minimal} if its degree is the minimum possible one.
\end{definition}

The following lemma is useful for us.

\begin{lemma}\label{every-point-is-good}
Let $E_*$ be a parabolic vector bundle over $X$ with parabolic divisor $D$. Take a
point $x\, \in\, X$. Then there is Kawamata cover which is good over some Zariski neighborhood
of $x$.
\end{lemma}

\begin{proof}
This follows from the construction of Kawamata cover \cite[Theorem 17]{Ka},
\cite[Theorem 1.1.1]{KMM} and the correspondence between parabolic bundles and equivariant
bundles \cite{Bi1}, \cite{Bo1}, \cite{Bo2}. The divisor $D'$ mentioned above moves
freely as it can be assumed to be very ample. This allows
us to have $D'' \cap U_x \,=\, D\cap U_x$ for suitable $D'$ and open neighborhood
$U_x$ of $x$. 
\end{proof}

It should be clarified that the covering in Lemma \ref{every-point-is-good} depends on the
point $x$.

\subsection{Parabolic bundles and ramified bundles}

One issue with the above correspondence between parabolic bundles and equivariant bundles is that the
ramified Galois covering $(Y,\, q)$ is not uniquely determined by the pair $(X,\, E_*)$. If
$(Y',\, q')$ is a ramified Galois covering of $X$ that factors through the covering $(Y,\, q)$, and the
map $Y'\,\longrightarrow\, Y$ is \'etale, then there is an equivariant vector bundle on $(Y',\, q')$ also
that corresponds to $E_*$. More generally, we can introduce extra divisors $D'$ on $X$ such that
$D\cup D'$ is still a simple normal crossing divisor, introduce trivial parabolic structure over $D'$, and
demand that the covering map ramified over $D'$ also. This non--uniqueness of the covering is addressed by
introducing what are known as ramified bundles, which we will briefly recall; more details can be found
in \cite[Section 2.2]{BD}, \cite[Section 3]{BL}, \cite{BBNR}.

Take any $(Y,\, q,\, V)$ corresponding to $E_*$, and as before denote $\text{Gal}(q)$ by $\Gamma$.
Let $\xi\, :\, F_V\, \longrightarrow\, Y$ be the holomorphic principal $\text{GL}(r,{\mathbb C})$--bundle defined by $V$,
where $r$ as before is the rank of $V$. So the group $\text{GL}(r,{\mathbb C})$ acts on $F_V$ holomorphically
and freely, and each fiber of $\xi$ is an orbit. The key point is that this action of
$\text{GL}(r,{\mathbb C})$ commutes with the action of
$\Gamma$ on $F_V$ given by the action of $\Gamma$ on $V$. Now consider the quotient
\begin{equation}\label{qr}
\widehat{\xi}\, :\, F'_V\, :=\, F_V/\Gamma\, \longrightarrow\, Y/\Gamma \,=\, X\, ,
\end{equation}
where $\widehat{\xi}$ is the descent of $\xi$. Since the action of $\text{GL}(r,{\mathbb C})$ and $\Gamma$
on $F_V$ commute, the quotient $F'_V$ is equipped with an action of $\text{GL}(r,{\mathbb C})$. This
action is free on $\widehat{\xi}^{-1}(X\setminus D)$, because for every $z\, \in\, q^{-1}(X\setminus D)$,
the action of $\Gamma_z$ on the fiber $V_z$ is trivial. This makes $F'_V\vert_{X\setminus D}$ a holomorphic principal
$\text{GL}(r,{\mathbb C})$--bundle over $X\setminus D$. However, over $\widehat{\xi}^{-1}(D)$, the action of
$\text{GL}(r,{\mathbb C})$ has finite isotropies.

A ramified principal $\text{GL}(r,{\mathbb C})$--bundle over $X$ with ramification over $D$ is defined
keeping the above model in mind. More precisely, a ramified principal $\text{GL}(r,{\mathbb C})$--bundle over
$X$ with ramification over $D$ consists of a smooth complex variety $E_{\text{GL}(r, {\mathbb C})}$
equipped with an algebraic right action of $\text{GL}(r, {\mathbb C})$
\begin{equation}\label{e2}
f\, :\,E_{\text{GL}(r, {\mathbb C})}\times \text{GL}(r,
{\mathbb C})\, \longrightarrow\, E_{\text{GL}(r, {\mathbb C})}\, ,
\end{equation}
and a surjective map
\begin{equation}\label{e1}
\widehat{\xi}\, :\, E_{\text{GL}(r, {\mathbb C})}\, \longrightarrow\, X
\end{equation}
such that the following conditions hold:
\begin{enumerate}
\item{} $\widehat{\xi}\circ f \, =\, \widehat{\xi}\circ p_1$, where $p_1$ is
the natural projection of $E_{\text{GL}(r, {\mathbb C})}\times
\text{GL}(r, {\mathbb C})$ to $E_{\text{GL}(r, {\mathbb C})}$,

\item{} for each point $x\, \in\, X$, the action of
$\text{GL}(r, {\mathbb C})$ on the
reduced fiber $\widehat{\xi}^{-1}(x)_{\text{red}}$ is transitive,

\item{} the restriction of $\widehat{\xi}$ to $\widehat{\xi}^{-1}(X
\setminus D)$ makes $\widehat{\xi}^{-1}(X\setminus D)$ a principal
$\text{GL}(r, {\mathbb C})$--bundle over $X\setminus D$ (note that the
first condition implies that $\widehat{\xi}^{-1}(X\setminus D)$ is preserved
by the action of $\text{GL}(r, {\mathbb C})$),

\item{} for each irreducible component $D_i\, \subset\, D$,
the reduced inverse image $\widehat{\xi}^{-1}(D_i)_{\text{red}}$ is a
smooth divisor and
$$
\widehat{D}\, :=\, \sum_{i=1}^\ell \widehat{\xi}^{-1}(D_i)_{\text{red}}
$$
is a normal crossing divisor on $E_{\text{GL}(r, {\mathbb C})}$, and

\item{} for any point $x$ of $D$, and any point
$z\, \in\, \widehat{\xi}^{-1}(x)$, the isotropy
group
\begin{equation}\label{e8}
G_z\, \subset\,\text{GL}(r, {\mathbb C}) \, ,
\end{equation}
for the action of $\text{GL}(r, {\mathbb C})$ on
$E_{\text{GL}(r, {\mathbb C})}$, is a finite group, and if
$x$ is a smooth point of $D$, then the natural action of 
$G_z$ on the quotient line $T_zE_{\text{GL}(r,
{\mathbb C})}/T_z\widehat{\xi}^{-1}(D)_{\text{red}}$ is faithful.
\end{enumerate}

The quotient in \eqref{qr} has all the above properties. Conversely, given any
ramified principal $\text{GL}(r,{\mathbb C})$--bundle over $X$ with ramification over $D$, there
is a parabolic vector bundle on $X$ of rank $r$ with parabolic structure on $D$. More precisely, we have
an equivalence of categories between the parabolic vector bundles on $X$ of rank $r$ with parabolic structure over $D$
and the ramified principal $\text{GL}(r,{\mathbb C})$--bundle over $X$ with ramification over $D$.

\section{Admissible Hermitian metric}\label{Admissible}

For the rest of the paper we assume 
that $D$ is smooth for the sake of convenience. Our results can be easily 
generalized to the case of simple normal crossing divisors.

Let $F$ be a $C^\infty$ complex vector bundle of rank $r$ over a complex manifold $Z$. Let $F_{\text{GL}(r)}\, 
\longrightarrow\, Z$ be the corresponding $C^\infty$ principal $\text{GL}(r,{\mathbb C})$--bundle. Giving a 
Hermitian structure on $F$ is equivalent to giving a $C^\infty$ reduction of structure group of $F_{{\rm 
U}(r)}\,\subset\, F_{\text{GL}(r)}$ to the subgroup ${\rm U}(r)\,\subset\, \text{GL}(r,{\mathbb C})$.

Let $\widehat{\xi}\, :\, E_{\text{GL}(r, {\mathbb C})}\, \longrightarrow\, X$ be a ramified principal
$\text{GL}(r,{\mathbb C})$--bundle over $X$ with ramification over $D$, as in \eqref{e1}. A \textit{Hermitian
structure} on $E_{\text{GL}(r, {\mathbb C})}$ (cf. \cite{BDey}) is a $C^\infty$ submanifold
$$
E_{\text{U}(r)}\, \subset\, E_{\text{GL}(r, {\mathbb C})}
$$
satisfying the following three conditions:
\begin{enumerate}
\item for the action of $\text{GL}(r, {\mathbb C})$ in \eqref{e2}, the submanifold $E_{\text{U}(r)}$
is preserved by $\text{U}(r)\, \subset\, \text{GL}(r, {\mathbb C})$,

\item for each point $x\, \in\, X$, the action of $\text{U}(r)$ on $E_{\text{U}(r)}\bigcap \widehat{\xi}^{-1}(x)$
is transitive, and

\item for each point $z\, \in\, \widehat{\xi}^{-1}(D)\bigcap E_{\text{U}(r)}$, the isotropy subgroup
for the action of $\text{GL}(r, {\mathbb C})$ for $z$ is contained in $\text{U}(r)$.
\end{enumerate}

A couple of comments are in order on the above definition. Take any $x\, \in\, D$. If the isotropy subgroup
of $\text{GL}(r, {\mathbb C})$ for some $z\, \in\, \widehat{\xi}^{-1}(x)\bigcap E_{\text{U}(r)}$
is contained in $\text{U}(r)$, then the isotropy subgroup for every point of 
$\widehat{\xi}^{-1}(x)\bigcap E_{\text{U}(r)}$ is contained in $\text{U}(r)$; this is because any two such isotropy
subgroups are conjugate by some element of $\text{U}(r)$. Since the isotropy subgroup for every $z\, \in\,
\widehat{\xi}^{-1}(D)$ is compact, a conjugate of it is contained in $\text{U}(r)$.

Let $E_*$ be a parabolic vector bundle on $X$ with parabolic structure over $D$. Let
$E_{\text{GL}(r, {\mathbb C})}$ be the corresponding ramified principal
$\text{GL}(r,{\mathbb C})$--bundle over $X$ with ramification over $D$.

\begin{definition}\label{adaptedbundle2}
An admissible Hermitian metric on $E_*$ is a smooth Hermitian metric $H$ on the
vector bundle $E\vert_{X\setminus D}$ such that the $C^\infty$ reduction of structure group
$$
E'_{\text{U}(r)}\, \subset\, 
E_{\text{GL}(r, {\mathbb C})}\vert_{X\setminus D}
$$
to the subgroup ${\rm U}(r)\,\subset\, \text{GL}(r,{\mathbb C})$ over $X\setminus D$ extends to
a Hermitian structure on the ramified principal bundle $E_{\text{GL}(r, {\mathbb C})}$.
\end{definition}

As above, let $E_*$ be a parabolic vector bundle of rank $r$ with
$\widehat{\xi}\, :\, E_{\text{GL}(r, {\mathbb C})}\, \longrightarrow\, X$
the corresponding ramified principal
${\text{GL}(r, {\mathbb C})}$--bundle. Let $V$ be a $\Gamma$--equivariant bundle 
on $Y$ that corresponds to $E_*$.
The $\Gamma$--equivariant holomorphic principal
${\text{GL}(r, {\mathbb C})}$--bundle on $Y$ associated to $V$ will be
denoted by $V_{\text{GL}(r, {\mathbb C})}$. Let
\begin{equation}\label{q0}
q_0\, :\, V_{\text{GL}(r, {\mathbb C})}\, \longrightarrow\, V_{\text{GL}(r, {\mathbb C})}/\Gamma
\,=\, E_{\text{GL}(r, {\mathbb C})}
\end{equation}
be the quotient map.

With the above set--up, we have the following simple but useful lemma.

\begin{lemma}\label{lem-am}
Every admissible Hermitian metric on $E_*$
is the descent of a unique $\Gamma$--invariant Hermitian structure on $V$.

Conversely, if $h$ is a $\Gamma$--invariant Hermitian structure on $V$ such that
the corresponding reduction of structure group $$V_{{\rm U}(r)}\,\subset\,
V_{{\rm GL}(r, {\mathbb C})}$$ to ${\rm U}(r)$ has the property that the quotient
$V_{{\rm U}(r)}/\Gamma$ is a $C^\infty$ submanifold of $V_{{\rm GL}(r, {\mathbb C})}/\Gamma\,=\,
E_{{\rm GL}(r, {\mathbb C})}$, then
$$
(V_{{\rm U}(r)}/\Gamma)\vert_{X\setminus D}\, \subset\,
E_{{\rm GL}(r, {\mathbb C})}\vert_{X\setminus D}
$$
is an admissible Hermitian metric on $E_*$.
\end{lemma}

\begin{proof}
Take an admissible Hermitian metric
$$
E'_{\text{U}(r)}\, \subset\, 
E_{\text{GL}(r, {\mathbb C})}\vert_{X\setminus D}
$$
on $E_*$. Let
$$
E_{\text{U}(r)}\, \subset\, E_{\text{GL}(r, {\mathbb C})}
$$
be the Hermitian structure on the ramified principal bundle $E_{\text{GL}(r, {\mathbb C})}$
obtained by extending $E'_{\text{U}(r)}$.
Now
$$
q^{-1}_0(E_{\text{U}(r)})\, \subset\, V_{\text{GL}(r, {\mathbb C})}
$$
is a $\Gamma$--invariant Hermitian structure on $V$, where $q_0$ is the
quotient map in \eqref{q0}. Uniqueness of the
$\Gamma$--invariant Hermitian structure on $V$ is evident.

To prove the converse, take a $\Gamma$--invariant Hermitian structure $h$ on $V$
satisfying the condition in the statement of the lemma. Let 
$$V_{\text{U}(r)}\,\subset\,
V_{\text{GL}(r, {\mathbb C})}$$
be the corresponding reduction of structure group to the subgroup ${\rm U}(r)\, \subset\,
\text{GL}(r, {\mathbb C})$. Since $h$ is $\Gamma$--invariant, the action of $\Gamma$ on
$V_{\text{GL}(r, {\mathbb C})}$ preserves the submanifold $V_{\text{U}(r)}$. So
$V_{\text{U}(r)}/\Gamma$ is equipped with an action of ${\rm U}(r)$.

Since $V_{\text{U}(r)}$ is preserved by the action of $\Gamma$, for any point $y\,\in\, Y$
and any $z\,\in\, (V_{\text{U}(r)})_y$, the orbit of $z$ under the action of the
isotropy subgroup $\Gamma_y\, \subset\, \Gamma$ is contained in $(V_{\text{U}(r)})_y$.
This implies that 
for each point $u\, \in\, \widehat{\xi}^{-1}(D)\bigcap E_{\text{U}(r)}$, the isotropy subgroup
for the action of $\text{GL}(r, {\mathbb C})$ for $u$ is contained in $\text{U}(r)$.
Consequently,
$$
(V_{\text{U}(r)}/\Gamma)\vert_{X\setminus D}\, \subset\,
E_{\text{GL}(r, {\mathbb C})}\vert_{X\setminus D}
$$
is an admissible Hermitian metric on $E_*$.
\end{proof}

Recall that the proof of Lemma \ref{every-point-is-good} is based on the fact that
the divisor $D'$ moves freely. This, combined with the proof
of Lemma \ref{lem-am}, gives the following:

\begin{lemma}\label{lem3}
Let $E_*$ be a parabolic vector bundle on $X$ with parabolic divisor $D$.
Let $H$ be a $C^\infty$ Hermitian metric on $E\vert_{X\setminus D}$ such that for for every
Kawamata covering
$$
q\, :\, Y\, \longrightarrow\, X
$$
for $E_*$, there is a ${\rm Gal}(q)$--invariant Hermitian structure $H'$ on the
${\rm Gal}(q)$--equivariant vector bundle $V$ corresponding to $E_*$ such that $H'$ descends to
$H$. Then $H$ is admissible.
\end{lemma}

In the next few paragraphs we discuss the above constructions from a concrete 
differentio--geometric point of view. For ease of exposition, in the rest of this section (unless specified otherwise) we 
will deal with good Kawamata covers.

As preparation, whenever we talk of a trivialization around a point on 
$D$, we consider ``adapted frames'', i.e., frames $\{e_1, e_2,\ldots\}$ on a neighbourhood in $X$ of 
a point in $D$ such that when restricted to $D$, the
collection $e_1,\ldots, e_{r_j}$ is a frame for $F_j$ 
($r_{j}-r_{j+1}$ is called the multiplicity of the weight $\alpha_j$).

Fix a good Kawamata cover
\begin{equation}\label{cp}
p\, :\, Y\, \longrightarrow\, X
\end{equation}
such that the parabolic bundle $E_{*}$ is the invariant direct image of an equivariant vector
bundle $V$ over $Y$. The Galois group $\text{Gal}(p)$ will be denoted by $\Gamma$. 

Assume that the adapted frame on $X$ is induced from a frame $\widetilde{e}_i$ on $Y$ such that the action of
$\Gamma$ in this frame is diagonal. As mentioned above, for point $y\, \in\,
p^{-1}(X\setminus D)$, the action of $\Gamma_y$ on $V_y$ is trivial, and $E\vert_{X\setminus D}$
is the descent of $V\vert_{p^{-1}(X\setminus D)}$. Therefore, any $\Gamma$--invariant Hermitian metric
on $V\vert_{p^{-1}(X\setminus D)}$ descends to a Hermitian metric on $E\vert_{X\setminus D}$, and conversely,
every Hermitian metric on $E\vert_{X\setminus D}$ is given by a unique
$\Gamma$--invariant Hermitian metric on $V\vert_{p^{-1}(X\setminus D)}$. Therefore, for the construction of a
singular Hermitian metric on $E$ singular over $D$ we may pretend that $D'$ is absent. In fact, our singular Hermitian metric
on $E$ will be given by a $\Gamma$ invariant Hermitian metric on $V$ for a covering $p$ with minimal degree
of ramification over $D$.

Cover $X\setminus D$ with coordinate open sets $U_{\gamma}$ with coordinates $z_{1,\gamma},z_{2,\gamma},
\ldots,z_{n,\gamma}$
such that $p^{-1}(U_{\gamma})$ is a coordinate open set on $Y$ with coordinates
$w_{1,\gamma},\ldots, w_{n,\gamma}$, the branched cover is given by $w_{1,\gamma} \,=\, z_{1,\gamma} ^{n_{\gamma}}$
(and $w_{i,\gamma} \,=\, z_{i,\gamma} \ \forall \ i\geq 2$), if $n_{\gamma}\,>\,1$ the
component $D$ of the branching divisor for $q$ is given by $z_1=0$, and $V$ is locally trivial over $p^{-1}(U_{\gamma})$. Denote neighbourhoods 
intersecting $D$ by $U_{D,\gamma}$. The sheaf $E_{*}$ on $U_{D,\gamma}$ is generated freely by $z_1 
^{\alpha_{r-j+1}} \widetilde{e}_{j}$. Therefore, the transition functions between 
$U_{D,\gamma}$ and $U_{D,\beta}$ are $g_{\gamma_D \beta_D}(z) \,=\, \mathrm{diag} 
[z_{1,\gamma}^{-\alpha_{r}}\ldots z_{1,\gamma}^{-\alpha_{1}} ]\widetilde{g}(\widetilde{z})_{\gamma \beta} 
\mathrm{diag}[z_{1,\beta} ^{\alpha_{r}} \ldots z_{1,\beta}^{\alpha_{1}}]$ where $\widetilde{g}$ are the 
transition functions of $V$ and $\widetilde{z}$ is the corresponding element in $p^{-1}(z)$. Likewise, 
$g_{\gamma_D \beta}(z)\,=\,\mathrm{diag}[z_{1,\gamma}^{-\alpha_{r}}\ldots z_{1,\gamma}^{-\alpha_1} 
]\widetilde{g}(\widetilde{z})_{\gamma \beta}$.

\begin{definition}
Suppose $E_{*}$ and $G_{*}$ are parabolic bundles with parabolic divisor $D$ on a
complex projective
manifold $X$, and flags, rational weights and transition functions $(F^E_j,\alpha_j=\frac{a_j}{N},\,g^E)$
and $(F^G_k,\beta_j=\frac{b_j}{N},\,g^G)$ respectively, where $N$ is the smallest such integer. Consider
the $N$--fold branched cover $p\,:\,Y\,\longrightarrow\, X$ ramified over a smooth reduced effective
divisor $D$ such that $E_{*}$ and $F_{*}$ are invariant direct images of bundles $V$ and $W$ over
$Y$ with transition functions $\widetilde{g}^V$ and $\widetilde{g}^W$. For the convenience of
differential geometers, we write the transition functions of the aforementioned constructions of
parabolic bundles below. On $X\setminus D$ they are given by their usual constructions. Therefore, we mention
only $g_{\gamma_D \beta}$ and $g_{\gamma_D \beta_D}$. 
\begin{center}
 \begin{tabular}{ | l | p{2cm} | p{2cm} | p{6cm} |}
 \hline
Bundle & Flag & Weights & Transition functions \\ \hline
 $E_{*}^{*}$ & $(F^E_j) ^{*}$ & $\alpha^{*}=1-\alpha_i$ if $\alpha_i >0$ and $0$ otherwise. & 	
{$\!\begin{aligned}
g^{E^{*}}_{\gamma_D \beta_D} &= \mathrm{diag} [z_{1,\gamma}^{-\alpha^{*}_{r}}\ldots z_{1,\gamma}^{-\alpha^{*}_1}]\\ &\times ((\widetilde{g} _{\gamma \beta})^T)^{-1} \mathrm{diag} [z_{1,\beta}^{\alpha_{r}^{*}}\ldots z_{1,\beta}^{\alpha^{*}_1}] \\
g^{E^{*}}_{\gamma_D \beta} &= \mathrm{diag} [z_{1,\gamma}^{-\alpha^{*}_{r}}\ldots z_{1,\gamma}^{-\alpha^{*}_1}] ((\widetilde{g} _{\gamma \beta})^T)^{-1}.
\end{aligned}$}
 \\ \hline
 $E_{*} \oplus G_{*}$ & $F^E _j \oplus F^G_k$ & $\{\alpha_j\} \cup \{\beta_k \}$ & $g = g^E \oplus g^G$. \\ \hline
$\det(E_{*})$ & $\det(E_{*}) \supsetneq 0$ & $\alpha^{*}= \sum \alpha_i \ \mathrm{mod} \ 1$ & $g^{\det(E)}_{\gamma_D \beta_D} = \det(\widetilde{g}_{\gamma \beta})\frac{z_{1,\beta}^{\alpha^{*}}}{z_{1,\gamma}^{\alpha^{*}}}$, and $g^{\det(E)}_{\gamma_D \beta} = \det(\widetilde{g}_{\gamma \beta})\frac{1}{z_{1,\gamma}^{\alpha^{*}}}$ \\ \hline 
 \end{tabular}
\end{center}
\label{usualconstructions}
\end{definition}

Note that the fractional powers $z_1 ^{\alpha}$ in the previous definition and the 
paragraph preceding that are to be taken on a suitably chosen branch of the 
complex plane.

Take any triple $(Y,\, \Gamma, \, V)$ as above such that the parabolic vector 
bundle $E_*$ coincides with the parabolic vector bundle associated to the 
$\Gamma$--linearized $V$. Then the projectivization $\mathbb{P}(E_{*})$ is defined 
to be the quotient $\mathbb{P}(V)/\Gamma$. Since for any $y\, \in\, p^{-1}(X\setminus D)$
the action of the isotropy subgroup $\Gamma_y\, \subset\, \Gamma$ on the fiber $V_y$ is trivial,
the quotient $\mathbb{P}(V)/\Gamma$ is a projective bundle over $X\setminus D$. More precisely,
the pulback of this bundle $$(\mathbb{P}(V)/\Gamma)\vert_{X\setminus D}\, \longrightarrow\, X\setminus D$$
to $Y\setminus p^{-1}(D)$ is canonically $\Gamma$--equivariantly identified with the projective bundle
$$\mathbb{P}(V)\vert_{Y\setminus p^{-1}(D)}\, \longrightarrow\, Y\setminus p^{-1}(D)$$ (any pullback
to $Y$ from $X$ is equipped with a tautological action of $\Gamma$).
The above quotient $\mathbb{P}(V)/\Gamma$ 
depends only on $E_*$ and it is independent of the choice of $(Y,\, \Gamma, \, 
V)$. There is a positive integer $m$ such that the isotropy groups, for the action 
of $\Gamma$ on $\mathbb{P}(V)$, act trivially on the fibers of 
$\mathcal{O}_{\mathbb{P}(V)}(m)$. Therefore, $\mathcal{O}_{\mathbb{P}(V)}(m)$ descends to the quotient 
$\mathbb{P}(V)/\Gamma\,=\, \mathbb{P}(E_{*})$ as a line bundle. Note that the discussion in this paragraph applies equally well even when $Y$ is a general (not necessarily good) Kawamata cover.

For differentio--geometric purposes, we need to study metrics and 
connections on both, the bundle, as well as the base manifold that respect the 
parabolic structure. For a given $E_*$ there may coverings as in \eqref{cp} giving $E_*$ such that the degree of the ramification over $D$ is arbitrarily large. We would be interested in coverings for which this degree is minimal. By a minimal branched cover of $X$ for $E_*$ we will mean a covering $p$ as in \eqref{cp} giving $E_*$ such that the degree of the ramification over $D$ is minimal. As in the previous paragraph, the Kawamata cover may be ramified over a divisor $D \cup D^{'}$ but the parabolic structure over $D^{'}$ is trivial. \\

\indent Now we look at the local picture of an admissible metric. Suppose we choose coordinates $(w_1,\ldots,w_n)$ on a good Kawamata cover $Y$ such that $w_1=0$ is the ramification divisor (assumed to be smooth), $z_1=w_1 ^N, \ z_i=w_i \ \forall \ i\geq 2$ is the quotient map near $D$, and a holomorphic frame $\widetilde{e}_1,\ldots, \widetilde{e}_r$ that is compatible with the flag. Note that the invariant direct image $E_{*}$ is locally freely generated by $e_{j}\,=\,\widetilde{e}_{j} w_1 ^{k_{r+1-j}}$ where $k_{j}$ are the weights of the action of $\Gamma$. 

\begin{lemma}
 The metric $H$ induced on $E_{*}$ is locally (near $D$) of the form $$H_{ij} (z) \,=\,
\overline{z}_1^{\alpha_{r+1-i}}\widetilde{H}_{ij}(w) z_1^{\alpha_{r+1-j}}\, .$$ (on a chosen branch of the complex plane). Conversely, if $H$ is a smooth metric on $E_{*}\vert_{X\setminus D}$ such that $$\widetilde{H}_{ij}(w)
\,=\,\overline{z}_1^{-\alpha_{r+1-i}} H_{ij} (z) z_1 ^{-\alpha_{r+1-j}}$$ extends smoothly (as a function of $w$) and positively across the branch cut then $H$ is induced from a $\Gamma$--invariant metric $\widetilde{H}$ on $V$. 
\label{local-admissibility-condition}
\end{lemma}

\begin{proof} Since $H_{ij}(z)\,=\, \langle e_i,e_j \rangle_{\widetilde{H}}$, we see that 
(using the physicist's convention for inner product) $$H_{ij} (z) \,=\,
\overline{z}_1^{\alpha_{r+1-i}}\widetilde{H}_{ij}(w) z_1 ^{\alpha_{r+1-j}}\, .$$

Conversely, given a metric $H$ on $E_{*}$ over $X\setminus D$, it induces an invariant metric $\widetilde{H}$ on $V$ over $Y\setminus \pi^{-1}(D)$. It follows from the
above observation that $$\widetilde{H}_{ij}(w)
\,=\,\overline{z}_1^{-\alpha_{r+1-i}} H_{ij} (z) z_1 ^{-\alpha_{r+1-j}}.$$ This clearly defines an invariant metric on $V$ because $$\widetilde{H}_{ij}(e^{\sqrt{-1}\theta}w_1,w_2,\ldots) = e^{(k_{r+1-i}-k_{r+1-j})\sqrt{-1}\theta} \overline{z}_1^{-\alpha_{r+1-i}} H_{ij}(z) z_1 ^{-\alpha_{r+1-j}} = e^{(k_{r+1-i}-k_{r+1-j})\sqrt{-1}\theta} \widetilde{H}_{ij}(w)$$
(and hence $\langle a,\, b\rangle_{\widetilde{H}}$ is invariant). Moreover, the
expression clearly does not depend on the branch cut chosen. So the only potentially
problematic points occur on $w_1=0$. But we know that $\widetilde{H}$ extends smoothly and hence defines a smooth invariant metric on $V$.
\end{proof}

Now we show a way to produce examples of admissible Hermitian metrics in the case where minimal good Kawamata covers exist.

\begin{lemma}
Let $H$ be induced by a $\Gamma_1$--invariant metric $\widetilde{H}_1$ on $V_1$ over a
minimal $N_1$-fold good Kawamata cover $Y_1$ over $X\,=\,Y_1/\Gamma_1$. Let $(Y_p,q_p, V_p)$ be
any locally good Kawamata covering of $X$ around $p$, with Galois group $\Gamma_2$ inducing the parabolic bundle $E_{*}$ on $X$.
Then $H\vert_{U_p}$ is induced by a smooth $\Gamma_2$--invariant metric $\widetilde{H}_2 \vert_{q_p^{-1}(U_p)}$ on $V_2\vert_{q_p^{-1}(U_p)}$.
\label{well-definedness-of-admissibility}
\end{lemma}

\begin{proof}
By minimality, $\frac{N_2}{N_1}=u$ is an integer $\geq 1$.

Suppose we choose coordinates $z,w$ near $D: z_1=0$ on $X$, $Y_1$
respectively (such that $z_1 =w_1 ^{N_1}$) and an admissible frame $e_i$ for $E_{*}$ induced from $Y_1$. Then 
\begin{gather}
(\widetilde{H}_1)_{ij} (w)=\overline{w}_1^{-k_{r+1-i}} H_{ij} (z(w)) w_1 ^{-k_{r+1-j}}.
\label{oldframe}
\end{gather}
Let $Z,W$ (such that $Z_1=W_1^{N_2}$) be coordinates near $D: Z_1=0$ on $X$, $Y_2$ respectively and an admissible frame $f_I$ for $E_{*}$ induced from $Y_2$. Clearly, $Z_1 = g(z) z_1$ where $g(z) \neq 0$ is a holomorphic function. Therefore, $w_1 = z_1 ^{1/N_1}=\frac{1}{g(z) ^{1/N_1}}W_1 ^u$, i.e., $w_1$ is a holomorphic function of $W_1$. Note that $H$ induces an invariant metric $\widetilde{H}_2$ outside the ramification divisor on $Y_2$. Let $f_I = t_{I}^{\ j}e_j$ (where we used the Einstein summation convention). By definition of admissibility, $t_{I}^{\ j}=0$ whenever $j>m(I)$ where $m(1),m(2)\ldots, m(r_{1})=r_1; m(r_{1}+1),\ldots,m_{r_2}=r_2,\ldots$. The expression for $\widetilde{H}_2$ near $D$ is given by
\begin{gather}
(\widetilde{H}_2)_{IJ} (W)=\overline{W}_1^{-k_{r+1-I}} H_{IJ} (Z(W)) W_1 ^{-k_{r+1-J}}
= \overline{Z}_1^{-\alpha_{r+1-I}} H_{IJ} (Z(W)) Z_1 ^{-\alpha_{r+1-J}}\nonumber \\
= \overline{g(z)} ^{-\alpha_{r+1-I}} g(z)^{-\alpha_{r+1-J}} \overline{z}_1^{-\alpha_{r+1-I}} H_{IJ} (Z(W)) z_1 ^{-\alpha_{r+1-J}} \nonumber \\
= \overline{g(z)} ^{-\alpha_{r+1-I}} g(z)^{-\alpha_{r+1-J}} \overline{z}_1^{-\alpha_{r+1-I}} H_{IJ} (Z(W)) z_1 ^{-\alpha_{r+1-J}} \nonumber \\
= \overline{g(z)} ^{-\alpha_{r+1-I}} g(z)^{-\alpha_{r+1-J}} \overline{z}_1^{-\alpha_{r+1-I}} t_I ^{\ i} \overline{t}_J ^{\ j}H_{ij} (Z(W)) z_1 ^{-\alpha_{r+1-J}}.
\label{newframe}
\end{gather}
As in Lemma \ref{local-admissibility-condition}, $\widetilde{H}_2$ can only be problematic on $W_1
\,=\,0$. Since the weights satisfy $\alpha_I \,\geq\, \alpha _J$ if $I\geq J$, and the expression
in \eqref{oldframe} is a smooth function of $w$ (which is in turn a smooth function of $W$), we see that 
$\widetilde{H}_2 (W)$ is smooth even at the origin.
\end{proof}

\begin{remark}\label{locally-good-local-coord}
Lemma \ref{local-admissibility-condition}, being a local condition, applies even to locally good Kawamata covers $Y_1$.
\end{remark}
\section{Positivity and ampleness}\label{positivity-and-ampleness}

Suppose $(Y,q,V)$ is a Kawamata cover of $(X,E_*,D)$. Then a $\Gamma$--invariant Hermitian metric on $V$ produces a $\Gamma$--invariant 
Hermitian metric on $\mathcal{O}_{\mathbb{P}(V)}(1)$, and hence a
$\Gamma$--invariant Hermitian metric on every $\mathcal{O}_{\mathbb{P}(V)}(n)$. Therefore,
if $\mathcal{O}_{\mathbb{P}(V)}(m)$ descends to $\mathbb{P}(E_{*})$, the descended line bundle
gets a Hermitian metric. 

It is known that the parabolic vector bundle $E$ is Hartshorne ample (respectively, Hartshorne nef)
if the vector bundle $V$ is Hartshorne ample (respectively, Hartshorne nef) \cite{Bi}, \cite{BN},
\cite{BL}.

Now we define Griffiths positivity of an admissible Hermitian metric.
\begin{definition}
Suppose $H$ is an admissible Hermitian metric on $E_*$. Then $H$ is said to be Griffiths (respectively, Nakano) positively curved if for every locally good Kawamata cover $(Y_p, q_p, V_p)$ around an arbitrary point $p$, the induced Hermitian metric $\widetilde{H}_p$ is so (in the usual sense of positivity) on $q^{-1}(U_p)$. 
\end{definition}

In particular, if an admissible metric is positively curved, then it is so in the usual sense on $X\setminus D$ because of Lemma \ref{every-point-is-good}. Moreover, for any Kawamata cover $(Y,q, V)$, it is easy to see that the induced smooth Hermitian metric $\widetilde{H}$ is Griffiths non--negatively curved in the usual sense with strict positivity outside the branching divisor.

\begin{remark}\label{easy-implication}
If there is a good Kawamata cover, then it is locally good over every open set and hence by definition the bundle $(V,\widetilde{H})$ is Griffiths positively curved. Therefore, $V$ is Hartshorne ample and hence $E$ is parabolic Hartshorne ample. However, if there is no good Kawamata cover, it is not obvious whether the parabolic Griffiths positivity of $(E_*, H)$ implies the Hartshorne 
ampleness of $V$ (which in turn implies the parabolic Hartshorne ampleness of $E$) or not. The best we can say directly is that $V$ is Hartshorne nef because the induced metric is positively curved away from the branching divisor. When there is no good Kawamata cover it is non-trivial (lemma \ref{compatibility-of-parabolic-and-non-parabolic-Griff-conj}) to prove that $V$ is actually Hartshorne ample.
\end{remark}

We have the following obvious analogue of the usual Griffiths conjecture.

\emph{Parabolic Griffiths conjecture} :
 There is an admissible metric $H$ such that $(E_*, H)$ is Griffiths positively curved if and only if $E_{*}$ is Hartshorne ample.

Before proceeding further, we define the notion of an admissible metric on $X$ with cone singularities 
on an effective reduced irreducible divisor $D$ with simple normal crossings.

\begin{definition}
If $(X,D)$ is a compact projective manifold with an effective smooth reduced divisor $D$, then
an $0\leq \alpha<2$--admissible Hermitian (respectively, K\"ahler) metric on $TX$ is a
Hermitian (respectively, K\"ahler) metric $\omega$ on $X\setminus D$ such that near $D$, in
coordinates $z_1, \ldots, z_n$ such that $D$ is $z_1\,=\,0$,
\begin{gather}
\omega-(\displaystyle \vert z_1 \vert^{-\alpha} dz_1 \wedge d\overline{z}_1 + \displaystyle \sum_{i=2}^n dz_i \wedge d\overline{z}_i) 
\label{admissible-Kahler-local-model}
\end{gather}
is smooth.
\label{admissible-Kahler-metric}
\end{definition}

\begin{remark}
Suppose $\alpha \,=\, 2-\frac{2}{N}$ for a positive integer $N$. If there is a good Kawamata $N$--fold branched cover $Y$ branched over the divisor $D$ (thus $X=Y/\Gamma$ where $\Gamma$ is a finite
group), then is not hard to see that an $\alpha$--admissible K\"ahler metric in this case is induced by a $\Gamma$--invariant K\"ahler metric on $Y$ and vice--versa. In the general case, suppose we cover $X$ by finitely many neighbourhoods $U_p$
that admit locally good Kawamata covers $(Y_p, \,q_p,\, V_p)$. Assume that $Y_p$ are endowed with K\"ahler metrics $\omega_p$. Then on $U_p$ these metrics induce admissible K\"ahler metrics $\omega_{p,\alpha}$. Using a partition--of--unity one can patch these to get an admissible Hermitian metric $\omega_{\alpha}$. \\
\indent Actually, every $(X,D)$ admits an $\alpha$--admissible K\"ahler metric if $X$ is compact K\"ahler. Indeed, take any smooth Hermitian metric $h$ on the bundle $[D]$. Suppose $\sigma$ is a defining section of $[D]$ and $\omega$ is a K\"ahler metric on $X$. Then $k\omega + \spbp \vert \sigma \vert_h^{2-\alpha}$ is an admissible K\"ahler metric for large enough $k$.
This metric induces K\"ahler metrics on $q_p^{-1}(U_p)$ for any locally good Kawamata cover $(Y_p,q_p)$.
\label{orbifold-Kahler}
\end{remark}

The following lemma reduces the problem of checking parabolic Griffiths positivity to usual Griffiths positivity on an open set.

\begin{lemma}
Suppose $N$ is the degree of any locally good minimal Kawamata cover of $(X,E_*,D)$. Suppose $H$ is an admissible Hermitian metric. Then $(E_{*},H)$ is parabolic Griffiths
positive if and only if on $X\setminus D$, $$\Theta_H \,\geq\, C \omega_{\alpha}$$ in the
Griffiths sense where $\omega_{\alpha}$ is any $\alpha\,=\,2-\frac{2}{N}$ admissible metric and $C>0$ is a constant.
\label{parabolic-Griffiths-positivity}
\end{lemma}

\begin{proof}
We recall that $(E_{*},H)$ is parabolic Griffiths positive if and only if the induced metric $\widetilde{H}_p$
is Griffiths positive on $U_p$ (and of course Griffiths non--negative on $Y_p$) for every locally good Kawamata cover and every point $p\in X$. Now, the Hermitian metric $\widetilde{H}_p$
is Griffiths positive if and only if $\Theta_{\widetilde{H}_p} \,\geq\, C \omega_p$ in the Griffiths sense where $\omega_p$ is any K\"ahler metric on $Y_p$. Using Lemma \ref{local-admissibility-condition} we compute the curvature $\Theta_H$ of $H$ in $U_p$ where $p\in D$ as follows:
\begin{gather}
(\Theta_H)_{ij} (z) \,=\, z_1 ^{-\alpha_{r+1-i}} (\Theta_{\widetilde{H}})_{ij}(w(z)) z_1 ^{\alpha_{r+1-j}}. 
\label{curvature-of-parabolic-bundle}
\end{gather} 
Now $\omega_p$ induces an $\alpha$--admissible metric $\omega_{p,\alpha}$ on $X$. The
inequality $\Theta_{\widetilde{H}_p} \geq C \omega_p$ when written in the $z$
coordinates is equivalent to
$$z_1 ^{-\alpha_{r+1-i}} (\Theta_H)_{ij} (z) z_1^{\alpha_{r+1-j}} \,\geq\, 
C \omega_{p,\alpha}\, .$$ Elsewhere this inequality is obvious. Therefore $\Theta_H \geq C \omega_{\alpha}$ where $\omega_{\alpha}$ is constructed by patching together $\omega_{p,\alpha}$. But since this holds for one admissible metric, it is easy to see that it does so for all. 
The converse follows by retracing the arguments above.
\end{proof}

The next lemma shows that the curvature of an admissible metric is a current for line bundles.

\begin{lemma}
If $(L_{*},h)$ is a parabolic line bundle with an admissible metric on
$X$, then the curvature $\Theta_h$ extends to
a closed current on $X$ that agrees with a smooth form outside
$D$. Moreover, $c_1(h) + \alpha [D]$ is an $L^1$ form (smooth outside $D$).
\label{curvature-current}
\end{lemma}

\begin{proof}
Obviously $\Theta_h$ is smooth outside $D$. In a neighbourhood of a point in $D$, let
$(Y_p, q_p, V_p, \widetilde{h}_p)$ be a locally good Kawamata cover. Then $h(z)\,=\,
\widetilde{h}_p(w)\vert z_1 \vert^{2\alpha}$. Therefore,
\begin{gather}
c_1(h) = -\alpha [D] + \frac{\sqrt{-1}}{2\pi}\overline{\partial}_z\partial_z \ln \widetilde{h}_p(w) \nonumber \\
= -\alpha [D] + \frac{\sqrt{-1}}{2\pi}\Theta_p(w(z))_{ij} dw^i \wedge dw^{\overline{j}}\nonumber \\
=-\alpha [D] + \frac{\sqrt{-1}}{2\pi}\Bigg [\frac{1}{N^2}\vert z_1 \vert^{2/N-2}\Theta(w(z))_{11} dz^1 \wedge d\overline{z}^1 + \displaystyle \sum_{i=2}^{n}\frac{1}{N}z_1 ^{1/N-1}\Theta(w(z))_{1i} dz^1 \wedge d\overline{z}^i \nonumber \\
 + \sum_{i=2}^{n} \frac{1}{N}\overline{z}_1 ^{1/N-1}\Theta(w(z))_{1i} dz^i \wedge d\overline{z}^1 + \sum_{i,j=2}^{n}\frac{}{}\Theta(w(z))_{ij} dz^i \wedge d\overline{z}^j\Bigg ] \nonumber \\
\Rightarrow c_1(h) + \alpha [D] = \frac{\sqrt{-1}}{2\pi}\Bigg [\frac{1}{N^2}\vert z_1 \vert^{2/N-2}\Theta(w(z))_{11} dz^1 \wedge d\overline{z}^1 + \displaystyle \sum_{i=2}^{n}\frac{1}{N}z_1 ^{1/N-1}\Theta(w(z))_{1i} dz^1 \wedge d\overline{z}^i \nonumber \\
 + \sum_{i=2}^{n} \frac{1}{N}\overline{z}_1 ^{1/N-1}\Theta(w(z))_{1i} dz^i \wedge d\overline{z}^1 + \sum_{i,j=2}^{n}\frac{}{}\Theta(w(z))_{ij} dz^i \wedge d\overline{z}^j\Bigg ]
\end{gather}
It is easy to see that $c_1(h) + \alpha [D]$ is an $L^1$ form. It is clearly closed
away from $D$. As a current, suppose $f$ is a smooth compactly supported
$(n-1,n-1)$--form in a coordinate neighbourhood $B$ of $D$; if $B_{\epsilon}$ is everything in $B$ outside a tubular neighbourhood of $D$ of size $\epsilon$, then
\begin{gather}
\displaystyle \int_{B_{\epsilon}} \pbp f (c_1(h) + \alpha [D]) = \frac{\sqrt{-1}}{2\pi}\int_{\partial B_{\epsilon}} \pbp f \partial \ln \widetilde{h}(w(z))
\end{gather}
The latter is easily seen to go to $0$ as $\epsilon \rightarrow 0$. Therefore $c_1(h) + \alpha[D]$ and hence $c_1(h)$ are closed currents.
\end{proof}

Now we prove that for line bundles, our metric notion of parabolic ampleness coincides with the algebro--geometric one in \cite{Bi}. 

\begin{lemma}
A parabolic line bundle $(L_{*},H)$ is ample in the metric sense above for some admissible metric $H$ if and only if $L_{*}$ is parabolic ample in the algebro--geometric sense. 
\label{parabolic-ample-lines}
\end{lemma}

\begin{proof}
A parabolic line bundle is algebro--geometrically parabolic ample if and only if 
$L+\alpha [D]$ is a K\"ahler class \cite{Bi}.

Suppose $L+\alpha D$ is a K\"ahler class $[\omega]$ and assume that $h_{L,0}$ and 
$h_{D,0}$ are metrics on $L$ and $[D]$ respectively. Then $$c_{1}(h_{L,0})+ \alpha 
c_1(h_{D,0}) \,=\, \omega+ \frac{1}{2\pi}\spbp \phi\, .$$ Defining a new metric 
$h_{L}\,=\,h_{L,0}e^{\phi}$ we see that $$c_1 (h_L)+\alpha c_1(h_D)\,=\,\omega\, .$$ Now define 
a singular metric on $L$ as $h_{par,L}=h_L \vert \sigma \vert_{h_{D,0}}^{\alpha}$ 
where $\sigma$ is the canonical section of $[D]$. Away from $D$, we have 
$c_1(h_{par,L})=\omega$. If $Y_p$ is any locally good minimal branched cover of $X$ around $p \in D$ and we choose 
coordinates $z$ on $X$ and $w$ on $Y$ so that near $p$, $w_1 ^N \,=\,z_1$, $w_i\,=\,z_i \ 
\forall \ i\geq 2$, then $\sigma \,=\,z_1$ and $\vert z_1 \vert^{-2\alpha} h_{par,L} \,=\, 
h_L h_{D,0}^{\alpha}$ which is a smooth positive function of $z$ and hence of $w$ 
in any neighbourhood of $w_1=0$. Thus by Lemma \ref{local-admissibility-condition} 
$(L_{*},h_{par,L})$ is parabolic ample in the metric sense.

Conversely, suppose $(L_{*},\,H)$ is parabolic ample in the metric sense. Then by Lemma \ref{parabolic-Griffiths-positivity} $\Theta_H \geq K\omega_{\alpha}$ away from $D$. Putting a singular metric $\frac{1}{\vert \sigma \vert^2}$ on $[D]$, we see that $L+\alpha [D]$ is represented by the current
$\widetilde{\omega}\,=\,\Theta_H+\alpha[D]$. This current is a positive measure from Lemma \ref{curvature-current}. It is easy to see (using the Bedford--Taylor \cite{BT} definition of products of currents) that the Nakai--Moishezon criterion is verified, i.e., that $(L+\alpha [D])^k. C >0$ for all $k$--dimensional subvarieties $C$. Indeed, if $C$ is not contained in $D$, then it is trivial because the current $\widetilde{\omega}\,>\,0$ and in $L^1$. If $C$ is contained in $D$, then locally $C= D \cap E$ where $E$ is another subvariety. Now approximating the currents $2\pi [D]=\spbp \ln \vert z_1 \vert^2, [E], \widetilde{\omega}$ by $D_{\epsilon}, E_{\epsilon}, \widetilde{\omega}_{\epsilon}$ (in the case of $[D]$ and $\widetilde{\omega}$, simply replacing $\vert z_1 \vert^2$ by $\vert z_1 \vert^2+\epsilon^2$) and taking limits we see that $(L+\alpha D)^k.C >0$. (The approximations weakly converge by Bedford--Taylor theory.) Hence $L+\alpha [D]$ is a K\"ahler class.
\end{proof}

The next lemma shows that parabolic Griffiths positivity implies parabolic Hartshorne ampleness 
in the metric sense. This is the ``easy'' direction of the Griffiths conjecture.

\begin{lemma}
Suppose $(E_*,H)$ is a parabolic bundle on $X$ with an admissible metric $H$. This induces an admissible metric $h$ on the parabolic bundle $\mathcal{O}_{\mathbb{P}(E_*)}(1)_{*}$
over $\mathbb{P}(E_*)$. Moreover, if $H$ is parabolic Griffiths positive, $E$ is parabolic Hartshorne ample.
\label{compatibility-of-parabolic-and-non-parabolic-Griff-conj}
\end{lemma}

\begin{proof}
Suppose $(Y,q,V)$ is a Kawamata cover of $(X,E_*,D)$. By definition, $H$ is the descent of a
smooth Hermitian metric $\widetilde{H}$ on $V$. Since $$\mathbb{P}(E_*)\,=\,\mathbb{P}(V)/\Gamma$$
and $\mathcal{O}_{\mathbb{P}(E_*)}(1)_{*}$ is the invariant direct
image of $\mathcal{O}_{\mathbb{P}(V)}(1)$, the smooth metric $\widetilde{h}$ that $\widetilde{H}$ induces on $\mathcal{O}_{\mathbb{P}(V)}(1)$ induces a unique metric $h$ downstairs on $\mathcal{O}_{\mathbb{P}(E_*)}(1)_{*}$. Since the cover $(Y,q,V)$ is arbitrary, $h$ is admissible (by
Lemma \ref{lem-am}). Moreover, since $\widetilde{h}$ is positively curved whenever
$\widetilde{H}$ is, by Lemma \ref{lem-am} and Lemma \ref{parabolic-Griffiths-positivity}, $h$ is
positively curved. Hence, by Lemma \ref{parabolic-ample-lines} we are done.
\end{proof}

Finally, we recall the notion of stability (and semistability) in the
parabolic case and recall the existence of Hermite--Einstein admissible metrics
in the parabolic setting when the weights $\alpha_i\,=\,\frac{k_i}{N}$ are rational.

Take a parabolic bundle $E_*$. For any coherent subsheaf $F\, \subset\, E$, the parabolic
structure on $E$ induces a parabolic structure on $F$. This induced parabolic structure will
be denoted by $F_*$. A parabolic bundle $E_*$ is called \textit{parabolic stable} (respectively,
\textit{parabolic semistable}) if for any coherent subsheaf $0\,\not=\, F\, \subsetneq\, F$ with $E/F$
torsionfree, we have
$$
\frac{\text{par-deg}(E_*)}{\text{rank}(E)} \, <\, \frac{\text{par-deg}(F_*)}{\text{rank}(F)}
\ \ \text{(respectively,\,~}\frac{\text{par-deg}(E_*)}{\text{rank}(E)} \,
\leq\, \frac{\text{par-deg}(F_*)}{\text{rank}(F)}{\rm )}\, .
$$

A parabolic vector bundle $E_*$ is called \textit{polystable} if
\begin{enumerate}
\item it is parabolic semistable, and 

\item it is a direct sum of parabolic stable bundles.
\end{enumerate}

Mehta and Seshadri proved that a parabolic vector bundle over a Riemann surface admits
a unitary flat connection if and only if it is polystable of parabolic degree zero \cite{MS}.
Biquard proved that a parabolic vector bundle over a Riemann surface admits
a Hermite--Einstein connection if and only if it is polystable \cite{Biq1}. This was generalized
to parabolic bundles on higher dimensional varieties in \cite{Biq}, \cite{BDey}.

Let $q\, :\, Y\, \longrightarrow\, X$ be a good Kawamata cover with Galois group
$\Gamma$, and $V\,\longrightarrow\, Y$ be a $\Gamma$--equivariant bundle
such that $E_*$ corresponds to $V$.

We will show that $E_*$ is semistable if and only if $V$ is semistable. First assume
that $E_*$ is not semistable. Let $F\, \subset\, E$ be a subsheaf that violates
the semistability condition. Then the subsheaf of $V$ generated by $q^*F$ such that
the quotient of $V$ by it is torsionfree contradicts the semistability condition for
$V$. Conversely, if $V$ is not semistable, consider the first term $W$ of the
Harder--Narasimhan filtration of $V$. From the uniqueness of the Harder--Narasimhan
filtration it follows that the subsheaf $W$ is preserved by the action of $\Gamma$ on $V$.
Hence the invariant direct image $(p_*W)^\Gamma$ is a subsheaf of $E$. This subsheaf
violates the semistability condition for $E_*$.

By an identical argument it follows that $E_*$ is polystable if and only if $V$ is 
polystable; we just need to replace the Harder--Narasimhan filtration by the socle 
filtration (see \cite[page 23, Lemma 1.5.5]{HL} for the socle filtration).

If $E_*$ is polystable, then consider the Hermite--Einstein connection on the
polystable vector bundle $V$. From the uniqueness of the Hermite--Einstein connection it 
follows that it is preserved by the action of $\Gamma$. Hence it descends to a
Hermite--Einstein connection on $E_*$. Conversely, a Hermite--Einstein connection on
$E_*$ produces an Hermite--Einstein connection on $V$, which implies that $V$ is
polystable. Hence $E_*$ is polystable. So $E_*$ is polystable if and only if it
admits an Hermite--Einstein structure.

\subsection{Parabolic Griffiths conjecture for Riemann surfaces}

Take a short exact sequence of parabolic bundles
\begin{equation}\label{abc}
0\, \longrightarrow\, A_*\, \longrightarrow\, B_* \, \longrightarrow\, C_*
\, \longrightarrow\, 0\, .
\end{equation}
We will show that
if $A_*$ and $C_*$ are Nakano positive then so is $B_*$. To prove this it suffices
to show that if
$$
0\, \longrightarrow\, A'\, \longrightarrow\, B' \, \longrightarrow\, C'
\, \longrightarrow\, 0\, .
$$
is a short exact sequence of equivariant vector bundles, and both
$A'$ and $C'$ admit equivariant Hermitian structures that are Nakano positive,
then $B'$ also admits an Hermitian structure which is Nakano positive. To prove that
$B'$ admits an Hermitian structure which is Nakano positive, observe that
Lemma 2.2 of \cite{Umemura} extends to equivariant set--up without any change.
Therefore, we conclude that $B_*$ is Nakano positive if $A_*$ and $C_*$ are so.

Armed with this observation about short exact sequences, we may prove the parabolic version of Griffiths' conjecture for Riemann surfaces.

\begin{theorem}\label{lemgr}
Take a parabolic vector bundle $E_*$ on a Riemann surface arising from a good Kawamata cover. Then the following two
are equivalent:
\begin{enumerate}
\item $E_*$ is parabolic ample.

\item $E_*$ is Nakano positive.
\end{enumerate}
\end{theorem}

\begin{proof}
First assume that $E_*$ is Nakano positive. Then every quotient of $E_*$ has
positive parabolic degree. This implies that $E_*$ is parabolic ample
\cite[Theorem 3.1]{Bi}.

Now assume that $E_*$ is parabolic ample. Therefore, every quotient of $E_*$ has
positive parabolic degree \cite[Theorem 3.1]{Bi}. Consider the Harder--Narasimhan
filtration of $E_*$. Every semistable parabolic vector
bundle on a Riemann surface admits a filtration of
subbundles such that each successive quotient is stable
of same parabolic slope. Using it we construct a finer filtration of the
Harder--Narasimhan filtration of $E_*$ such that each successive
quotient is stable and the parabolic slopes are decreasing (now the slope is no longer
strictly decreasing). Next observe that a stable parabolic bundle of positive
parabolic degree is Nakano positive because the Hermitian--Einstein metric on it
is Nakano positive. On the other hand, $E_*$ is a successive extensions by
stable parabolic bundles of positive degree starting with a stable parabolic
bundle of positive degree. Therefore, from the earlier observation that
$B_*$ in \eqref{abc} is Nakano positive if $A_*$ and $C_*$ are so, we now
conclude that $E_*$ is Nakano positive.
\end{proof}

\section{Chern--Weil theory for parabolic bundles}\label{ChernWeil}

In this section, we develop Chern--Weil theory for admissible Hermitian metrics and provide 
further evidence for Griffiths' conjecture, by proving that the push--forward of 
powers of the first Chern form of the tautological bundle on the parabolic 
projectivization are the Segre forms downstairs. (Our proof is a calculation 
involving generating functions and is hence technically slightly different from 
the analogous ones for usual bundles by \cite{Diverio} and \cite{Guler} .) In 
addition, for surfaces, akin to \cite{Pinchern} we prove that parabolic ample 
stable bundles (arising from good Kawamata covers) admit metrics whose Schur forms are positive.

First we prove a lemma to the effect that invariant closed forms and cohomology classes descend from $Y$ to $X$.

\begin{lemma}
Suppose $Y$ is branched cover of $X$ ramified over a smooth effective reduced divisor $D\subset X$. If $\widetilde{\eta}$ is a smooth closed invariant $k$--form on $Y$, then it descends to a closed current $\eta$ on $X$ which is a smooth form away from $D$. If $\widetilde{\eta}\,=\,d\widetilde{\gamma}$ where $\widetilde{\gamma}$ is smooth and invariant, then $\eta=d\gamma$ where $\gamma$ is a current on $X$ that is smooth away from $D$.
\label{invariant-closed-forms}
\end{lemma}

\begin{proof}
By invariance, $\widetilde{\eta}$ induces a smooth closed form $\eta$ away from $D$ on
$X$. (Likewise, if $\widetilde{\eta}\,=\,d\widetilde{\gamma}$, away from $D$, we have a smooth
form $\gamma$.) Suppose $\widetilde{\eta}$ is a $(p,q)$--form locally (near $D$) given by
$$\widetilde{\eta}(w) 
\,=\, \widetilde{\eta}_{IJ}(w)dw^I \wedge d\overline{w}^J\, ,$$ where $I,J$ are multi--indices and $w$ are coordinates on $Y$ chosen such that $z_1=w_1^N$ where $z$ are coordinates on $X$. Define 
\begin{gather}
\eta(z)=\displaystyle \sum_{1\notin I,1 \notin J}\widetilde{\eta}_{IJ}(w(z))dz^I d\overline{z}^J + \sum_{J=(1,J^o), 1\notin I}\frac{\widetilde{\eta}_{I1J^o}(\overline{z}^1)^{1/N-1}}{N} (w(z)) dz^{I} d\overline{z}^1 d\overline{z}^J \nonumber \\
+\sum_{I=(1,I^o), 1\notin J}\frac{\widetilde{\eta}_{1I^o J}(w(z)) (z_1)^{1/N-1}}{N}dz^1 dz^{I^o} d\overline{z}^J+\sum_{I=(1,I^o), J=(1,J^o)}\widetilde{\eta}_{1I^o J}(w(z))\frac{\vert z^1\vert^{2/N-2}}{N^2}dz^1 dz^{I^o} d\overline{z}^1 d\overline{z}^{J^o}
\label{downstairs-expression}
\end{gather}
Suppose $f$ is a smooth $(n-p,n-q)$--form with compact support in the given coordinate neighbourhood $B$ of $D$. Denote the region in $B$ outside an $\epsilon$--tubular neighbourhood of $B$ by $B_{\epsilon}$ 
\begin{gather}
\displaystyle \int_{B_{\epsilon}} df \wedge \eta = \int_{\partial B_{\epsilon}} f\wedge \eta
\,\rightarrow\, 0
\end{gather}
as $\epsilon \rightarrow 0$. Therefore $d\eta =0$ as a current. If $\widetilde{\eta}\,=\,
d_w \widetilde{\gamma}$, and hence $\eta=d_z\gamma$ away from $D$ (technically, away from a branch cut, but this will play no role because $\gamma$ is well--defined as a smooth form on $X
\setminus D$ by invariance of $\widetilde{\gamma}$). Now
\begin{gather}
\displaystyle \int_{B_{\epsilon}}f \wedge \eta\,=\, -\int_{B_{\epsilon}}df \wedge \gamma + \int_{\partial B_{\epsilon}} f\wedge \gamma \,\rightarrow\, -\int_{B_{\epsilon}}df \wedge \gamma
\end{gather}
as $\epsilon\,\rightarrow \,0$ because by invariance, $\gamma$ is given by similar expression as
\eqref{downstairs-expression}.
\end{proof}

\begin{remark}\label{descent-is-local}
Note that actually lemma \ref{invariant-closed-forms} applies locally as well, i.e., when $X, Y$ are non-compact.
\end{remark}

It is easy to see that the Chern--Weil forms of an invariant metric on an equivariant bundle $V$ over $Y$ are invariant differential forms. Now we may define the parabolic Chern--Weil forms of an admissible metric.

\begin{definition}
Suppose $(E_{*},D,H)$ is a parabolic bundle on a smooth projective variety $X$ with an admissible metric $H$ and $(Y_p, q_p, V_p)$ is a locally good Kawamata cover around $p \in X$. 
Let $\widetilde{H}_p$ be the smooth metric on the equivariant bundle $V_p$ descending to $H$. Then, given any invariant 
polynomial $\Phi$ acting on matrices, the parabolic Chern--Weil form of $\Phi(E_{*},H)(p)$ at the point $p$
 is the form $\Phi_{par}(\Theta_H)(p)$ induced from the 
invariant forms $\Phi(\Theta_{\widetilde{H}_p})(p)$.
\label{parabolic-chern-weil}
\end{definition}

We will prove the following lemma.

\begin{lemma}
The parabolic Chern--Weil currents depend only on $H$ and not on the locally good Kawamata covers
$Y_p$. Moreover, the cohomology classes of parabolic Chern--Weil currents
$[\Phi_{par}(\Theta_H)]$ are independent of the admissible metric chosen to define them.
\label{well-definedness-of-classes} 
\end{lemma}

In order to use Lemma \ref{invariant-closed-forms} in this situation, we need to prove that a Bott--Chern form of an invariant metric is an invariant form. 

\begin{lemma}
Suppose a finite group $\Gamma$ acts by biholomorphisms on a complex manifold
$Y$ and its action lifts to an action on a holomorphic vector bundle $V$ over $Y$
of rank $r$. Assume that $\widetilde{H}_1, \widetilde{H}_2$ are two smooth Hermitian
metrics on $Y$. Also assume that $\Phi$ is an invariant polynomial on matrices. Then
there exists an invariant Bott--Chern form $\widetilde{\Phi}(\widetilde{H}_2,
\widetilde{H}_1)$ satisfying $$\frac{\sqrt{-1}}{2\pi} \pbp
\widetilde{\Phi}(\widetilde{H}_2,\widetilde{H}_1) \,=\,
\Phi(\widetilde{H}_2)-\Phi(\widetilde{H}_1)\, .$$
\label{invariant-bott-chern}
\end{lemma}

\begin{proof}
The construction we use here is due to Gillet--Soul\'e \cite{GS}. Consider the vector
bundle $\widetilde{V} \,=\, \pi_1^{*}V\otimes \pi_2^{*}\mathcal{O}_{\mathbb{P}^1}(1)$ over
$Y\times \mathbb{P}^1$. If $$i_p \,:\, Y \,\longrightarrow\, Y \times \mathbb{P}^1$$ is
the inclusion map $x\,\longmapsto\, (x,p)$, then $i_p^{*}(\widetilde{V}) \equiv V$ for
all $p\in \mathbb{P}^1$. Extend the action of $\Gamma$ to $Y\times \mathbb{P}^1$ by making it act
trivially on the second factor. This also lifts to an action to $\widetilde{V}$. Take an affine open cover
$U_0, U_1$ of $\mathbb{P}^1$. Since this is a trivializing open cover for
$\mathcal{O}_{\mathbb{P}^1}(1)$, we may define a Hermitian metric $$\widetilde{H} \,=\,
(1-\rho) \widetilde{H}_1 + \rho \widetilde{H}_2$$ on $\widetilde{V}$, where
$\rho,\, 1-\rho$ is a partition of unity subordinate to the open cover such that
$\rho=1$ on a neighbourhood of $0$. This is clearly an invariant metric. Therefore
$$\widetilde{\Phi}(H_2,H_1)\,=\,
\displaystyle \int_{\mathbb{P}^1}\Phi(\widetilde{V},\widetilde{H})\ln\vert z \vert^2$$
is the desired Bott--Chern form which is also invariant. 
\end{proof}

Now we are in a position to prove Lemma \ref{well-definedness-of-classes}.

\begin{proof}[{Proof of Lemma \ref{well-definedness-of-classes}}]
According to \eqref{curvature-of-parabolic-bundle}, away from a branch cut near $D$,
we have $\Theta_H = z_1^{-\alpha} \Theta_{\widetilde{H}_p} (z_1^{\alpha})^{-1}$ and hence
$$\Phi_{par}(\Theta_H) (p)\,=\,\Phi (\Theta_{\widetilde{H}_p})(p)\, .$$ Elsewhere this equality is
obvious. Therefore, the parabolic Chern--Weil currents depend only on $H$ and not the
specific cover used. If $H_1, H_2$ are two admissible metrics induced from
$\widetilde{H}_{1,p}$, $\widetilde{H}_{2,p}$ on $V_p$ over $Y_p$, Lemma
\ref{invariant-bott-chern} shows that
$$\frac{\sqrt{-1}}{2\pi}\pbp \widetilde{\Phi}_p(\widetilde{H}_2,\widetilde{H}_1) (q)\,
= \,\Phi(\widetilde{H}_2)(q)-\Phi(\widetilde{H}_1)(q) \ \forall \ q\in U_p$$ and that the Bott--Chern form $\widetilde{\Phi}_p$ can be chosen to be invariant. The proof of Lemma \ref{invariant-bott-chern} shows that $\widetilde{\Phi}_p$ thus constructed is actually independent of $p$. Therefore, using Lemma \ref{invariant-closed-forms} we see that the induced parabolic Chern--Weil currents have a unique cohomology class.
\end{proof}

This allows us to define the parabolic Chern classes as the cohomology classes 
$[c_{par}(\Theta_H)]$ where $H$ is any admissible metric. These cohomology classes 
coincide with the ones defined in \cite{Bi3}, \cite{IS}, \cite{BD}. The first step to proving 
this is the following theorem (which in the usual case was proven in 
\cite{Guler,Diverio}).

\begin{theorem}
Suppose $(E_*,D,H)$ is a parabolic bundle on $X$ with an admissible metric $H$. Let $h$ be the induced admissible metric on the parabolic bundle $L_*
\,=\,\mathcal{O}_{\mathbb{P}(E_*)}(1)_{*}$ over $\mathbb{P}(E_*)$. Let $s(E_*,H)$ and $s(L_*,h)$ be the Segre polynomial currents (inverses of the Chern polynomials) of $(E_*,H)$ and $(L_*,h)$ respectively. Then, we have the following inequality of smooth forms on $X\setminus D$:
\begin{gather}
\pi_{*}s(L_*,h)\,=\,s(E_*,H)\, ,
\label{pushforward-equality-of-forms}
\end{gather}
where $\pi_*$ is the push--forward (fiber integral) of forms. Moreover, if $f$ is a smooth form with compact support on $X$, then 
\begin{gather}
\displaystyle \int_X f \wedge s(E_*,H) = \int_{\mathbb{P}(E_*)} \pi^{*} f \wedge s(L_*,h).
\label{pushforward-currents}
\end{gather}
\label{guler-type-theorem}
\end{theorem}

\begin{proof}
Suppose $(Y_p, q_p, V_p)$ is a locally good Kawamata cover around $p \in X\setminus D$. Let $(\widetilde{L}=\mathcal{O}_{\mathbb{P}(V_p)}(1),\,\widetilde{h})$ be a Hermitian holomorphic line bundle over $\mathbb{P}(V_p)$ inducing $(L_{*},h)$ over $X$. The results of \cite{Guler, Diverio} show that equation
\eqref{pushforward-equality-of-forms} holds for $\widetilde{h}$ near $q_p^{-1}(p)$. By equivariance, this implies that the equation holds true for $h$ as well. We shall give a proof of this result of \cite{Guler, Diverio} here. A small technical advantage of this proof is that it uses generating functions and hence has the potential to produce more such formulae (in the context of general flag varieties). In what follows, we denote the projection map from $\mathbb{P}(V)$ to $Y$ as $\pi$. Moreover, we choose a trivialization $U_p\times\mathbb{P}^r$ of $\mathbb{P}(V)$ around $p$ such that $U_p$ is a coordinate neighbourhood and the frame is a normal frame, i.e., the metric at $p$ is Euclidean up to second order. We also evaluate the fiber integrals over the affine chart $w_i=\frac{X_i}{X_0}$:
\begin{gather}
\displaystyle \int_{\pi^{-1}(p) \subset \mathbb{P}(V)} s(\widetilde{L},\widetilde{h}) = \int_{\pi^{-1}(p) \subset \mathbb{P}(V)} \frac{1}{1+c_1(\widetilde{L},\widetilde{h})} \nonumber \\
= \int_{\mathbb{C}^{r-1}} \frac{1}{1+\frac{\sqrt{-1}}{2\pi}\Big(\sum_{i,j}\frac{((1+\vert \vec{w} \vert^2) \delta_{ij}-\overline{w}_i w_j)dw_i \wedge d\overline{w}_j}{(1+\vert \vec{w} \vert^2)^2}+\frac{\Theta(
\widetilde{H})_{00}+\sum_i (\Theta_{i0}(\widetilde{H})w_i + \Theta_{0i}(\widetilde{H})\overline{w}_i)+\sum_{i,j}(\Theta(\widetilde{H}))_{ij}w_i \overline{w}_j}{1+\sum_i \vert w_i \vert^2}\Big)} \nonumber \\
= f(\Theta)
\label{guler-proof-1}
\end{gather}
where it is easy to see that $f(\Theta)$ is a universal polynomial (does not depend on $\Theta$) with rational coefficients in the entries of the matrix of $2$-forms $\Theta$. Therefore, we may assume without loss of generality that actually, $\Theta$ is simply a skew-Hermitian matrix of complex numbers. It is also easy to see (a change of trivialization does not change the Chern forms) that $f$ is an invariant polynomial. Hence we may (without loss of generality) assume that $\Theta = \mathrm{diag}(a_0,a_1,a_2,\ldots, a_r)$ where $a_i =\sqrt{-1}b_i$ are purely imaginary numbers. Now we may evaluate $f(\Theta)$ easily :
\begin{gather}
f(\Theta) = \int_{\mathbb{C}^{r-1}} \frac{1}{1+\frac{\sqrt{-1}}{2\pi}\Big(\sum_{i,j}\frac{((1+\vert \vec{w} \vert^2) \delta_{ij}-\overline{w}_i w_j)dw_i \wedge d\overline{w}_j}{(1+\vert \vec{w} \vert^2)^2}+\frac{a_0+\sum_{i}a_i \vert w_i \vert^2}{1+\vert \vec{w} \vert^2}\Big)} \nonumber \\
= \int_{\mathbb{C}^{r-1}} \frac{1}{1+\frac{\sqrt{-1}}{2\pi}\Big(\frac{a_0+\sum_{i}a_i \vert w_i \vert^2}{1+\vert \vec{w} \vert^2} \Big)}\frac{1}{1+\frac{\frac{\sqrt{-1}}{2\pi}\Big(\sum_{i,j}\frac{((1+\vert \vec{w} \vert^2) \delta_{ij}-\overline{w}_i w_j)dw_i \wedge d\overline{w}_j}{(1+\vert \vec{w} \vert^2)^2}\Big)}{1+\frac{\sqrt{-1}}{2\pi}\Big(\frac{a_0+\sum_{i}a_i \vert w_i \vert^2}{1+\vert \vec{w} \vert^2} \Big)}} \nonumber \\
= (-1)^{r-1} \int_{\mathbb{C}^{r-1}} \frac{1}{1+\frac{\sqrt{-1}}{2\pi}\Big(\frac{a_0+\sum_{i}a_i \vert w_i \vert^2}{1+\vert \vec{w} \vert^2} \Big)} \Bigg (\frac{\frac{\sqrt{-1}}{2\pi}\Big(\sum_{i,j}\frac{((1+\vert \vec{w} \vert^2) \delta_{ij}-\overline{w}_i w_j)dw_i \wedge d\overline{w}_j}{(1+\vert \vec{w} \vert^2)^2}\Big)}{1+\frac{\sqrt{-1}}{2\pi}\Big(\frac{a_0+\sum_{i}a_i \vert w_i \vert^2}{1+\vert \vec{w} \vert^2} \Big)} \Bigg )^{r-1}\nonumber \\
= (-1)^{r-1} (r-1)!\Big (\frac{\sqrt{-1}}{2\pi} \Big )^{r-1} \int_{\mathbb{C}^{r-1}} \Bigg( \frac{1}{1+\frac{\sqrt{-1}}{2\pi}\Big(\frac{a_0+\sum_{i}a_i \vert w_i \vert^2}{1+\vert \vec{w} \vert^2} \Big)} \Bigg)^r \Bigg (\frac{1}{1+\vert \vec{w} \vert^2} \Bigg )^r dw_1 \wedge d\overline{w}_1 \ldots \nonumber \\
= (-1)^{r-1} (r-1)!\Big (\frac{\sqrt{-1}}{2\pi} \Big )^{r-1} \int_{\mathbb{C}^{r-1}} \Bigg( \frac{1}{c_0+\sum_{i}c_i \vert w_i \vert^2 } \Bigg)^r dw_1 \wedge d\overline{w}_1 \ldots,
\label{guler-proof-2} 
\end{gather}
where $c_i = \frac{\sqrt{-1}}{2\pi}a_i+1$. We evaluate the last integral as follows.
\begin{gather}
f(\Theta) = \frac{(-1)^{r-1} (r-1)!}{c_0^r}\Big (\frac{\sqrt{-1}}{2\pi} \Big )^{r-1}\int_{\mathbb{C}^{r-1}} \Bigg( \frac{1}{1+\sum_{i}\frac{c_i}{c_0} \vert w_i \vert^2 } \Bigg)^r dw_1 \wedge d\overline{w}_1 \ldots \nonumber \\
= \frac{1}{c_0 c_1 c_2 \ldots c_{r-1}} = s(V,\widetilde{H}),
\label{guler-proof-3}
\end{gather}
where the second-to-last equality follows from a simple change of variables and the fact that $c_1 (\mathbb{P}^{r-1})^{r-1}\,=\,1$.

Now suppose $f$ is a smooth compactly supported form and $X_{\epsilon}$ is everything in $X$ outside of an $\epsilon$--tubular neighbourhood of $D$. Then
\begin{gather}
\displaystyle \int_X f\wedge s(E_*,H) = \lim_{\epsilon \rightarrow 0}\int_{X_{\epsilon}} f\wedge s(E_*,H) \nonumber \\
= \lim_{\epsilon \rightarrow 0} \int_{\pi^{*}(X_{\epsilon})\subset \mathbb{P}(E_{*})} \pi^{*} f \wedge s(L_*,h) \nonumber \\
= \int_{\mathbb{P}(E_*)} \pi^{*} f \wedge s(L_*,h)
\label{guler-proof-4}
\end{gather}
where the second-to-last equality follows from \eqref{pushforward-equality-of-forms} and the limits follow from the proofs of Lemma \ref{invariant-closed-forms} and Lemma \ref{well-definedness-of-classes}. Indeed, the Chern-Weil currents are actually $L^1$ forms.
\end{proof}

Let $E_*$ be a parabolic vector bundle on $X$. Let $Y$ be a Kawamata cover of
$(X,E_*,D)$ and $V$ an equivariant vector bundle on $Y$, such that $E_*$ corresponds to $V$. The
$i$--th parabolic Chern class of $E_*$ is the push--forward of the $i$--th Chern class
of $V$ \cite{Bi3}, \cite{IS}, \cite{BD}. Therefore, Theorem \ref{guler-type-theorem} has the following
corollary:

\begin{corollary}\label{corcc}
The parabolic Chern classes defined earlier using admissible metrics coincide with the
parabolic Chern classes defined in \cite{Bi3}, \cite{IS}, \cite{BD}.
\end{corollary}

\begin{remark}
The usual proof (see \cite{Gulerthesis} for instance) that $c_1(\Theta), c_2(\Theta)>0$ for Griffiths positive bundles shows that this holds even for parabolic Griffiths positive bundles (as weakly positive currents). The proof in \cite{Guler} shows that positivity is preserved under fiber integral and hence the signed parabolic Segre forms are positive. Hence, on surfaces, if $E_*$ is parabolic Griffiths positive, then $c_{1,par}, c_{2,par}, c_{1,par}^2-c_{2,par}>0$ as currents.
\label{admissible-positive-forms}
\end{remark}

Given Remark \ref{admissible-positive-forms}, it is but natural to ask whether a parabolic Hartshorne ample bundle admits an admissible metric whose Schur polynomial currents are weakly positive. In the usual (non--parabolic) case, this was proven for semistable bundles on surfaces. Here we prove an analogous result for parabolic stable bundles induced from \emph{good} Kawamata covers. Prior to that we define the notion of an admissible form :

\begin{definition}
Given an integer $N$, suppose $Y$ is an $N$--fold branched cover of $X$ branched over a divisor $D \subset X$. An $N$--admissible $(k,k)$--form $\eta$ is the $L^1$ current induced from a smooth form $\widetilde{\eta}$ on $Y$. It is said to be weakly positive if $\widetilde{\eta}$ is weakly positive on $Y$. Note that this definition is applicable locally too.
\label{admissible-forms-definition}
\end{definition}

\begin{remark}
Lemma \ref{local-admissibility-condition} 
shows that the notion of weak positivity does not depend on the cover $Y$ chosen and that an admissible $(n,n)$--form is $\geq C \omega_{2-2/N}^N$ on $X\setminus D$ where $\omega_{2-2/N}$ is induced from a smooth K\"ahler form of a Hermitian metric $\omega$ on $Y$.
\label{well-definedness-of-admissible-forms}
\end{remark}

\begin{theorem}
If $E_*$ induced from a minimal good Kawamata cover $(Y,q,V)$ is parabolic Hartshorne ample and parabolic stable with respect to an admissible K\"ahler
class $[\omega_{2-2/N}]$ on a compact complex surface $X$, then $E_*$ admits an admissible metric $G$ such that
$$c_{1,par}(G)\,>\,0\, ,\ \ c_{2, par}(G)\,>\,0\, ,\ \ c_{1,par}^2-c_{2,par}\,>\, 0\, .$$
\label{parabolic-stable-surfaces}
\end{theorem}

\begin{proof}
Fix an admissible K\"ahler metric $\omega_{\alpha}$ arising from a smooth K\"ahler metric on a minimal $N$--fold Kawamata cover $Y$ (with covering group $\Gamma$).

In \cite{Bi3} a result akin to Bloch--Gieseker \cite{BG} was proven for parabolic Chern classes of parabolic ample 
bundles. Therefore $\displaystyle \int_X (c_{1,par}^2-c_{2,par}) >0$. This means that the cohomology class 
$[c_{1,par}^2-c_{2,par}]$ admits an admissible positive representative $\eta \geq C\omega_{\alpha}^2$, i.e., $\eta$ 
arises from an invariant smooth non--negative form $\widetilde{\eta}$ on $Y$. Since $E_*$ is stable with respect to the 
$(2-2/N)$--admissible K\"ahler class $[\omega_{2-2/N}]$, it admits an admissible Hermite--Einstein metric $H$. Now $H$ 
induces a smooth invariant Hermite--Einstein (with respect to the induced K\"ahler metric $\widetilde{\omega}$) 
$\widetilde{H}$ on $V$ over $Y$.

We wish to find a smooth invariant function $\widetilde{\phi}$ on $V$ such that $\widetilde{G}=\widetilde{H}e^{-\phi}$ satisfies
$$c_1(\widetilde{G})\,>\,0\, ,$$ and
\begin{gather}
(c_{1}^2-c_{2})(\widetilde{G}) \,=\, \widetilde{\eta}\, .\label{monge-ampere} 
\end{gather}
If we manage to do so, then the calculations in \cite{Pinchern} show that $c_{2}(\widetilde{G})>0$. This would complete the proof of Theorem \ref{parabolic-stable-surfaces}. Indeed, just as in \cite{Pinchern}, equation reduces to the following Monge--Amp\`ere equation.
\begin{gather}
\frac{r(r+1)}{2} \left (\ddc \widetilde{\phi} + \frac{c_1 (\widetilde{H})}{r} \right )^2 = \eta + \frac{2r c_2(\widetilde{H})-(r-1)c_1^2(\widetilde{H})}{2r} \label{explicit-monge-ampere}
\end{gather}
The right hand side of \eqref{explicit-monge-ampere} is positive owing to the Kobayashi--L\"ubke inequality (see \cite{Pinchern} for details). Moreover $[c_1 (V)]$ is a K\"ahler class admitting an invariant K\"ahler metric (indeed take the push--forward of $c_1(\mathcal{O}_{\mathbb{P}(V)}(1))^{r}$ over $\mathbb{P}(V)$). Thus \eqref{explicit-monge-ampere} admits an invariant smooth solution $\widetilde{\phi}$. It admits a smooth solution thanks to Yau's proof of the Calabi conjecture \cite{Yau}. The uniqueness of the solution makes it invariant because the K\"ahler class and the right hand side are invariant under the action of $\Gamma$. This induces a continuous function $\phi$ on $X$ smooth outside $D$ such that the induced metric $G$ on $E_*$ is admissible and satisfies the desired properties.
\end{proof}
\section{Curvature of direct images}\label{Directimages}

In this section, we prove an analogue of Berndtsson's theorem \cite{Bo} for equivariant (i.e.,
parabolic) bundles. In particular, this implies that if $E_{*}$ is Hartshorne ample, then 
$E_{*}\otimes \det(E)_{*}$ is Nakano positive.

\begin{theorem}
Suppose $X$ is a compact complex manifold equipped with a holomorphic action of a finite
group $\Gamma$. Assume that the action of $\Gamma$ lifts to an action on a holomorphic line
bundle $L$ over $X$ and that $h$ is a $\Gamma$--invariant Hermitian metric on $L$ with semipositive
curvature. Suppose $X$ admits a holomorphic submersion $\pi :X \longrightarrow \widetilde{X}$
to a compact complex manifold $\widetilde{X}$ such that $\widetilde{X}$ admits an action of
$\Gamma$ that commutes with the action on $X$.

There exists a holomorphic vector bundle $E$ over $\widetilde{X}$ with a lift of the action of $\Gamma$ such that for each $n$--dimensional fiber $X_t \,=\, \pi^{-1}(t)$, where $t\,\in \,\widetilde{X}$, the fiber $E_t$ of $E$ over $t$ is the vector space of holomorphic sections of $L\vert_{X_t} \otimes K_{X_t}$ on $X_t$. Moreover, the following metric $H$ on $E$ is invariant under $\Gamma$ and has Nakano semipositive curvature:
\begin{gather}
\vert u_t \vert ^2 \,=\,
\displaystyle \int _{X_t} \sqrt{-1}^{n^2} u_t \wedge \overline{u}_t h
\,=\, \int _{X_t} \sqrt{-1}^{n'} u_t \wedge \overline{u}_t h
\, ,
\label{metriconE}
\end{gather}
where $n'$ is $0$ or $1$ depending on whether $n$ is even or odd. Moreover, if $h$ has strictly positive curvature in a neighbourhood of $\pi^{-1}(p)$ where $p\in \widetilde{X}$, then $H$ is Nakano (strictly) positively curved near $p$.
\label{Berndtssontheorem}
\end{theorem}

\begin{proof}
The fact that $E$ exists as a vector bundle and that $H$ has Nakano semipositive curvature follows 
from theorem 1.2 in \cite{Bo}. The proof in \cite{Bo} also shows the strict positivity statement. All we need to do is to check that the action of $\Gamma$ lifts 
to $E$ and that $H$ is invariant under the same.

Firstly, the action of $\Gamma$ extends to $(p,q)$--forms by means of pullback, i.e., $g. \omega 
\,=\, (f_{g^{-1}})^{*}\omega$. By equivariance, $L_t \otimes K_{X_t}$ is taken to $L_{g.t} \otimes 
K_{X_{g.t}}$. Thus $\Gamma$ acts on sections of $L_t \otimes K_{X_t}$ over $X_t$ and takes to 
them to sections of $L_{g.t} \otimes K_{X_{g.t}}$ over $X_{g.t}$. Consequently, $E$ admits a lift of the 
action of $\Gamma$. By invariance of $h$ it is easy to see that the metric \ref{metriconE} is an 
invariant metric.
\end{proof}

Theorem \ref{Berndtssontheorem} implies the following result. This result 
provides further evidence for the parabolic version of Griffiths' conjecture.

\begin{corollary}
A parabolic Hartshorne ample bundle $E_*$ arising from a minimal good Kawamata cover $(Y,q,V)$ over a projective manifold $X$ admits an admissible metric $H$ that induces an admissible parabolic Nakano positively curved metric on $E_* \otimes \det(E_*)$. 
\label{nakanopositivity}
\end{corollary}

\begin{proof}
By parabolic Hartshorne ampleness, $\mathcal{O}_{\mathbb{P}(V)}(1)$ admits a positively curved invariant metric. Thus, $V$ admits an invariant Hermitian metric $\widetilde H$ such that the Hermitian metric on $V\otimes \det (V)$ induced by $\widetilde H$ is Nakano positive. The induced metric on $E_{*}\otimes \det(E_*)$ is admissible because $Y$ is a minimal good Kawamata cover (lemma \ref{well-definedness-of-admissibility}). It is of course Nakano positively curved in the parabolic sense.
\end{proof}

\section*{Acknowledgements}
The authors thank the anonymous referee for useful suggestions to improve the content and the presentation of the paper significantly.
The work of the second author (Pingali) was partially supported by SERB grant No. 
ECR/2016/001356. The second author also thanks the Infosys foundation for the Infosys young 
investigator award. The second author is also partially supported by grant F.510/25/CAS-II/2018(SAP-I) from UGC (Govt. of India). The first author is supported by a J. C. Bose Fellowship.

\end{document}